\documentclass{amsart}
\usepackage{amsmath}
\usepackage{mathrsfs}
\usepackage{amsfonts}
\usepackage{amssymb,color}
\usepackage{graphicx}
\usepackage[square,numbers]{natbib}
\usepackage{verbatim}
\usepackage{enumerate}
\usepackage{float}
\usepackage{hyperref}
\newtheorem{theorem}{Theorem}
\newtheorem{claim}[theorem]{Claim}

\newtheorem{lemma}[theorem]{Lemma}
\newtheorem{proposition}[theorem]{Proposition}

\newtheorem*{theorem*}{Theorem}
\newtheorem*{lemma*}{Lemma}

\theoremstyle{definition}
\newtheorem{definition}[theorem]{Definition}

\theoremstyle{remark}

\numberwithin{theorem}{section}
\numberwithin{equation}{section}

\def\Xint#1{\mathchoice
	{\XXint\displaystyle\textstyle{#1}}%
	{\XXint\textstyle\scriptstyle{#1}}%
	{\XXint\scriptstyle\scriptscriptstyle{#1}}%
	{\XXint\scriptscriptstyle\scriptscriptstyle{#1}}%
	\!\int}
\def\XXint#1#2#3{{\setbox0=\hbox{$#1{#2#3}{\int}$}
		\vcenter{\hbox{$#2#3$}}\kern-.5\wd0}}

\def\dashint{\Xint-}

\newcommand{\dd}{\; \mathrm{d}}

\newcommand{\bbG}{\mathbb{G}}

\newcommand{\bbR}{\mathbb{R}}
\newcommand{\bbN}{\mathbb{N}}

\DeclareMathOperator*{\lip}{Lip}

\begin{document}
	\title[Maximal Directional Derivatives in Laakso Space]{Maximal Directional Derivatives in Laakso Space}
	
	\author[Marco Capolli]{Marco Capolli}
	\address[Marco Capolli]{Institute of Mathematics, Polish Academy of Sciences, Jana i Jadrzeja Sniadeckich 8, Warsaw, 00-656, Poland}
	\email[Marco Capolli]{Mcapolli@impan.pl}
	
	\author[Andrea Pinamonti]{Andrea Pinamonti}
	\address[Andrea Pinamonti]{Department of Mathematics, University of Trento, Via Sommarive 14, 38123 Povo (Trento), Italy}
	\email[Andrea Pinamonti]{Andrea.Pinamonti@unitn.it}
	
	\author[Gareth Speight]{Gareth Speight}
	\address[Gareth Speight]{Department of Mathematical Sciences, University of Cincinnati, 2815 Commons Way, Cincinnati, OH 45221, United States}
	\email[Gareth Speight]{Gareth.Speight@uc.edu}
	
	\date{\today}
	
	\begin{abstract}
		We investigate the connection between maximal directional derivatives and differentiability for Lipschitz functions defined on Laakso space. We show that  maximality of a directional derivative for a Lipschitz function implies differentiability only for a $\sigma$-porous set of points. On the other hand, the distance to a fixed point is differentiable everywhere except for a $\sigma$-porous set of points. This behavior is completely different to the previously studied settings of Euclidean spaces and Carnot groups.
	\end{abstract}
	
	\maketitle

\section{Introduction}
	
Rademacher's theorem states that each Lipschitz function between Euclidean spaces is differentiable almost everywhere with respect to Lebesgue measure. Hence, while Lipschitz functions are relatively flexible, they still have strong differentiability properties. This important result has many consequences. For instance, it is used to prove the area and coarea formulas and to study rectifiable sets \cite{EG15, Mag12}.

Rademacher's theorem has been further studied in several directions. One direction of research extends Rademacher's theorem to Lipschitz functions between more general spaces. There are versions of Rademacher's theorem for mappings between infinite dimensional Banach spaces \cite{BL00, LPT13}, Carnot groups \cite{Pan89,MPS,PS172}, and metric measure spaces admitting a differentiable structure \cite{Che99}. Interesting features arise in each case. In infinite dimensional Banach spaces one must distinguish between Gateaux differentiability (directional derivatives form a linear map) and full Frechet differentiability (difference quotients also converge uniformly). In Carnot groups, the derivatives are group linear mappings. In metric measure spaces, including the Laakso space studied in the present paper, the notion of differentiability is with respect to a collection of Lipschitz charts.

Another direction of research investigates to what extent Rademacher's theorem is optimal. Rademacher's theorem can equivalently be stated as follows. Whenever a Lipschitz mapping $f\colon \bbR^{n}\to \bbR^{m}$ fails to be differentiable at every point of a set $N\subset \bbR^{n}$, then the set $N$ must have Lebesgue measure zero. The converse question asks if $N\subset \bbR^{n}$ has Lebesgue measure zero, must there exist a Lipschitz map $f\colon \bbR^{n}\to \bbR^{m}$ which fails to be differentiable at every point of $N$? The answer to this question is yes if and only if $n\leq m$ and combines the work of several authors \cite{ACP10, CJ15, PS15, Pre90, Zah46}. In particular, if $n>m$ then there exists a measure zero set $N\subset \bbR^{n}$ with the following property. For every Lipschitz map $f\colon \bbR^{n}\to \bbR^{m}$ there exists a point $x\in N$ such that $f$ is differentiable at $x$. Such a set $N$ is called a universal differentiability set. The size of such sets have been more widely studied in the case $m=1$. In particular, they can be made compact and Hausdorff or Minkowski dimension one \cite{DM11, DM12, DMf, DM14}.

The key technique underlying the construction of universal differentiability sets is the fact that in Euclidean spaces (and some other settings) maximality of a directional derivative for a Lipschitz function implies differentiability. More precisely, if $f\colon \mathbb{R}^{n}\to \bbR$ is a Lipschitz function and $|f'(x,e)|=\mathrm{Lip}(f)$ for some $x\in \bbR^{n}$ and $e\in \bbR^{n}$ with $|e|=1$, then $f$ is differentiable at $x$ \cite{Fit84}. Such a fact is also important in proving several differentiability results in infinite dimensional Banach spaces \cite{LPT13}. In \cite{LPS17, PS16, PS18}, the second and third authors extended this fact and the study of universal differentiability sets to Lipschitz maps $f\colon \bbG\to \bbR$ where $\bbG$ is a Carnot group. There they showed that the implication maximality implies differentiability holds for directional derivatives $Ef(x)$ in a horizontal direction $E$ at a point $x\in \bbG$ if and only if the distance to the origin is differentiable at the point reached by following the direction $E$. Notice maximality implies differentiability is independent of the point $x\in \bbG$ due to the group translations. These equivalent statements are true for every direction in Carnot groups of step two, but more generally depend on the Carnot group and the direction considered. In particular, there exists a Carnot group where they fail for every direction. These techniques were used to show that measure zero universal differentiability sets exist in every Carnot group of step two and in families of Carnot groups of arbitrarily high step.

The present paper investigates to what extent the connection between maximality implies differentiability and differentiability of the distance holds in the nonlinear setting. We focus our attention on the Laakso space \cite{Laa00}, one of the best known examples of a metric measure space which admit a differentiable structure but not a linear structure. Laakso space is defined beginning with $I\times K$ where $I=[0,1]$ and $K$ is a  Cantor set. One identifies components of $K$ at suitable heights (jump levels) in $I$ in order to make a path connected space $F=(I\times K)/ \sim $ (see Section \ref{Preliminaries} for more details). The distance $d$ is then the path distance between points. This space was introduced by Laakso \cite{Laa00} to show that there exist metric measure spaces which are Ahlfors $Q$-regular and support a Poincar\'e inequality for any $Q>1$. As a consequence of \cite{Che99}, such a space admits a differentiable structure.

Since Laakso space is a doubling metric measure space supporting a Poincar\'e inequality \cite{Laa00}, it admits a differentiable structure with respect to which Lipschitz functions are differentiable almost everywhere \cite{Che99}. The differentiable structure consists of a single chart with projection onto $I$ (Definition \ref{differentiability}). This fact seems well known to experts but we were unable to find an explicit reference. Hence we give an explicit proof in Section \ref{RademacherProof}. Next we define directional derivatives by considering difference quotients along line segments in the $I$ direction (Definition \ref{def_vert_deriv}). This is natural since geodesics consist of line segments in the $I$ direction with at most a countable number of jumps. We show that, as in the Euclidean and Carnot setting, the Lipschitz constant of a Lipschitz map $f\colon F\to \bbR$ is the supremum of directional derivatives $f_{I}(x)$ (Proposition \ref{Lipissup}) over points $x\in F$. This motivates the definition of maximal directional derivative given in Definition \ref{defM}. With these ingredients in place, we now state our first main theorem.	
	
	\begin{theorem}\label{mainthmmaxdiff}
		Let $M$ be the set of $x\in F$ such that whenever $f\colon F\to \bbR$ is Lipschitz with directional derivative $f_{I}(x)=\pm \lip(f)$ at $x$, then $f$ is differentiable at $x$. 
		
		Then $M=\{[x_{1},x_{2}]\colon x_{1}\in S\}$.  In particular $M$ is $\sigma$-porous.
	\end{theorem}
	
The set $S$ consists of those heights in $I$ which see jump levels roughly equidistant above and below on all sufficiently small scales (see Definition \ref{defS} for details). The most striking part of the theorem is that the set where maximality implies differentiability is $\sigma$-porous. A set is porous if it has relatively large holes on arbitrarily small scales and $\sigma$-porous if it is a countable union of $\sigma$-porous sets (Definition \ref{porous}). Such sets are extremely small in multiple ways. Every $\sigma$-porous set is of first category and of measure zero. In summary, the set of points where the implication maximality implies differentiability holds is very small.
	
We next look at differentiability of the distance. There is no distinguished origin in Laakso space and the geometry around different points can look quite different. Hence we study differentiability of the map $y\mapsto d_{p}(y):=d(y,p)$ for each fixed point $p\in F$. Our second main result is the following.
	
\begin{theorem}\label{distancetheorem}
Let $p\in F$ and denote by $B_p\subset F$ the set of points in $F$ at which $d_p$ is not differentiable. Then $h(B_{p})$ is countable, in particular $B_p$ is $\sigma$-porous.	
\end{theorem}
	
Theorem \ref{distancetheorem} states that the set of points where the distance is differentiable is extremely large (complement of $\sigma$-porous set). On the other hand, Theorem \ref{mainthmmaxdiff} states that the set of points where maximality implies differentiability holds is extremely small (it is $\sigma$-porous). Laakso space cannot be written as a union of $\sigma$-porous sets. Hence maximality implies differentiable cannot be equivalent to differentiability of the distance as in Euclidean spaces and Carnot groups.

We now summarize the organization of the paper. In Section \ref{Preliminaries} we give the basic definitions. In Section \ref{Maximaldirectional} we prove some basic facts about directional derivatives and prove Theorem \ref{mainthmmaxdiff}. In Section \ref{distdiff} we prove Theorem \ref{distancetheorem}. Finally in Section \ref{RademacherProof} we prove Theorem \ref{Rademacher} which verifies the notion of differentiability we have considered is the natural one.
	
\bigskip
	
	\textbf{Acknowledgements:} Part of this paper was written while M. Capolli was a PhD student at the University of Trento advised by the other two authors. A. Pinamonti is a member of {\em Gruppo Nazionale per l'Analisi Matematica, la Probabilit\`a e le loro Applicazioni} (GNAMPA) of the {\em Istituto Nazionale di Alta Matematica} (INdAM). G. Speight  was  supported by a grant from the Simons Foundation (\#576219, G. Speight). Part of this work was done while G. Speight was visiting the University of Trento and supported by funding from the University of Trento.
	
	 The techniques and results in Section \ref{RademacherProof} should be known to experts. Some of the ideas used in that section come from the paper \cite{BS13} by David Bate and the third author and from discussions of the third author with David Bate and David Preiss.
	
	\section{Preliminaries}\label{Preliminaries}
	
	\subsection{Laakso Space}
	
	Let $I=[0,1]$ and let $K\subset [0,1]$ the standard middle third Cantor set. We define $K_{0}:=(1/3)K$ and $K_{1}:=(1/3)K+(2/3)$ to be the left and right similar copies of $K$. We then define $K_{00}:=(1/3)K_0=(1/9)K$ and $K_{01}:=(1/3)K_1=(1/9)K+(2/9)$ to be the left and right similar copies of $K_{0}$. The set $K_{a}$ is defined similarly when $a$ is any finite string of $0$'s and $1$'s. 
	
	We define the height of a point $(x_{1}, x_{2})\in I\times K$ by $h(x_{1},x_{2}):=x_{1}$. If $n\in \bbN$ and $m_{i}\in \{0,1,2\}$ for $1\leq i\leq n$, we define $w(m_{1},\ldots, m_{n}):=\sum_{i=1}^{n} m_{i}/3^{i}$.
	A wormhole level of order $n$ is a set of the form
	\[\{w(m_{1},\ldots, m_{n})\} \times K \subset I\times K, \qquad m_{n}>0.\]
	The condition $m_{n}>0$ implies wormhole levels of different orders do not overlap. We denote the set of wormholes of order $n$ by $J_{n}$.
	
	\begin{definition}
		We define an equivalence relation $\sim$ on $I\times K$ as follows. For each $n\in \bbN$ and wormhole level $\{w(m_{1},\ldots, m_{n})\} \times K$ of order $n$, identify pairwise $\{w(m_{1},\ldots, m_{n})\} \times K_{a0}$ and $\{w(m_{1},\ldots, m_{n})\} \times K_{a1}$ for each binary string $a$ of length $n-1$. More precisely, a point $(x_1, x_2)\in \{w(m_{1},\ldots, m_{n})\} \times K_{a0}$ is identified with $(x_1, x_2+(2/3^n)) \in \{w(m_{1},\ldots, m_{n})\} \times K_{a1}$. Such an identified point is called a wormhole of order $n$.
	\end{definition}
	
	Define $F:=(I\times K)/\sim$. Let $q\colon I\times K\to F$ be given by $q(x_1, x_2)=[x_1, x_2]$, where $[x_{1}, x_{2}]$ denotes the equivalence class in $F$ of $(x_{1},x_{2}) \in I\times K$. We define the height $h\colon F\to I$ by $h[x_{1},x_{2}]=x_{1}$. Notice this is well defined because points identified in the construction of $F$ have the same coordinate in $I$. We define a metric $d$ on $F$ by
	\[d(x,y)=\inf \{\mathcal{H}^{1}(p)\colon q(p) \mbox{ is a path joining }x \mbox{ and }y\},\]
	where $p\subset I\times K$. In \cite{Laa00} it is shown that any pair of points can be connected by a path and so the metric $d$ is well defined. For $p\in F$ we denote by $d_{p}$ the map from $F$ to $\bbR$ given by $y\mapsto d_{p}(y):=d(y,p)$. Clearly this map is $1$-Lipschitz for any $p\in F$. The following proposition gives information about geodesics \cite[Proposition 1.1]{Laa00}.
	
	\begin{proposition}\label{prop_geodesic}
		Fix $x,y\in F$ with $h(x)\leq h(y)$. Let $[a,b]\subset I$ be an interval of minimum length that contains the heights of $x$ and $y$ and all the wormhole levels needed to connect those points with a path. Let $p$ be any path starting from $x$, going down to height $a$, then up to height $b$, then down to $y$. 
		
		Then $p$ is a geodesic connecting $x$ and $y$. All geodesics from $x$ to $y$ are of that form for some interval $[a',b']$ such that $b'-a'=b-a$.
	\end{proposition}
	
	We will call the interval $[a,b]$  in Proposition \ref{prop_geodesic} a \emph{minimal height interval} for $x$ and $y$. If $h(x)\geq h(y)$ we use the same terminology, but the geodesic will begin at $y$, then go down to height $a$, then up to height $b$, then down to $x$. The following Proposition is {\cite[Proposition 1.2]{Laa00}}. It relates minimal height intervals to the distance between points.
	
	\begin{proposition}\label{prop_distance}
		Let $x,y\in F$ with a minimal height interval $[a,b]$. Then
		\[
		d(x,y) = 2b-2a-|h(x)-h(y)|.
		\]
	\end{proposition}
	
	Let $Q:=1+ (\ln 2/\ln 3)$. It is shown in \cite{Laa00} that $F$ is Ahflors $Q$-regular with respect to the metric $d$. This means there exists a constant $C\geq 1$ such that
	\[C^{-1}R^{Q}\leq \mathcal{H}^{Q}(B(x,R))\leq CR^{Q}\]
	for all $x\in F$ and $R>0$.
	
	\begin{definition}
		Laakso space is the set of equivalence classes $F:=(I\times K)/\sim$ equipped with the metric $d$ and Hausdorff dimension $\mathcal{H}^{Q}$.
	\end{definition} 
	
	\subsection{Differentiability of Functions on the Laakso Space}
	
	We now define what we mean by directional differentiability and differentiability.
	
	\begin{definition}\label{def_vert_deriv}
		Let $f\colon F\to \bbR$ and $x=[x_{1},x_{2}]\in F$. 
		
		Suppose $x$ is not a wormhole. Whenever the limit exists, we define
		\begin{equation}\label{limit_vert_deriv}
			f_{I}(x):=\lim_{t\to 0} \frac{f[x_{1}+t,x_{2}]-f[x_{1},x_{2}]}{t}.
		\end{equation}
		The limit is one-sided if $x_{1}=0$ or $1$.
		
		Suppose $x$ is a wormhole of order $n$ and $(x_{1},x_{2})\in I\times K$ is the representative of $x$ with the smaller value of $x_{2}$. Whenever the limit exists, we define
		\begin{align*}
			f_{L}(x)&:=\lim_{t\to 0} \frac{f[x_{1}+t,x_{2}]-f[x_{1},x_{2}]}{t}\\
			f_{R}(x)&:=\lim_{t\to 0} \frac{f[x_{1}+t,x_2+(2/3^n)]-f[x_{1},x_2+(2/3^n)]}{t}.
		\end{align*}
		If $f_{L}(x)$ and $f_{R}(x)$ exist and are equal, we say that $f_{I}(x)$ exists and define it to be the common value. The limits are one-sided if $x_{1}=0$. 
	\end{definition}
	
	\begin{definition}\label{differentiability}
		Let $f\colon F\to \bbR$ and $x\in F$. We say that $f$ is differentiable at $x$ if there exists $Df(x)\in \bbR$ such that
		\[ \lim_{y\to x} \frac{f(y)-f(x)-Df(x)(h(y)-h(x))}{d(y,x)}=0.\]
	\end{definition}
	
	We will study the relationship between directional derivatives and differentiability in Laakso space. The following definition of maximal directional derivatives is motivated by Proposition \ref{Lipissup}.
	
	\begin{definition}\label{defM}
		Let $f\colon F\to \bbR$ be Lipschitz and $x\in F$. Suppose $f_{I}(x)$ exists and $|f_{I}(x)|= \mathrm{Lip}(f)$. Then we say that $f$ has a \emph{maximal directional derivative} at $x$.
		
		We define $M$ to be the set of $x\in F$ for which the following implication holds true. Suppose a Lipschitz map $f\colon F\to \bbR$ has a maximal directional derivative at $x$. Then $f$ is differentiable at $x$.
	\end{definition}
	
	The following definition will be helpful in investigating $M$. Recall $\inf \varnothing=\infty$.
	
	\begin{definition}\label{defS}
		For any $t\in (0,1)$ and $n\in \bbN$, we define
		\begin{equation}\label{dplus}
			D_{n}^{+}(t):=\inf \{s>0: t+s\in J_{n} \},
		\end{equation}
		\begin{equation}\label{dminus}
			D_{n}^{-}(t):=\inf \{s>0: t-s\in J_{n} \}.	
		\end{equation}
		We define $S$ to be the set of $t\in (0,1)$ for which there is $C(t)\geq 1$, $N(t)\in \bbN$ so that
		\begin{equation}\label{fractionpm}
			C(t)^{-1}\leq \frac{D_{n}^{+}(t)}{D_{n}^{-}(t)} \leq C(t) \qquad \mbox{for } n\geq N(t).
		\end{equation}
	\end{definition}
	
	For any $t\in (0,1)$, $D_{n}^{+}(t)$ and $D_{n}^{-}(t)$ are non-zero and are finite for sufficiently large $n$. Notice that $S\neq \varnothing$; for instance $J_{n}\subset S$ for all $n\geq 1$.
	
	Laakso space is a PI space, so admits a differentiable structure of charts with respect to which Lipschitz functions are almost everywhere differentiable \cite{Che99, Laa00}. It seems to be understood by researchers in the field that one can choose a single chart consisting of the whole Laakso space together with the height map, giving the definition of differentiability in Definition \ref{differentiability}. However, we were unable to find any clear proof in the literature. Since these ideas are important motivation for the present paper, we justify this by proving the following theorem in the Section \ref{RademacherProof}.
	
	\begin{theorem}\label{Rademacher}
		Every Lipschitz function $f\colon F\to \bbR$ is differentiable almost everywhere.
	\end{theorem}
	
	We do not claim that Theorem \ref{Rademacher} is new or that the proof is original. Some of the ideas for the proof used come from the paper \cite{BS13} and from discussions of the third author with David Bate and David Preiss.
	
	\subsection{Porous Sets}
	
	We now define porous sets. They provide a way to say a set is small or exceptional in a very strong sense.
	
	\begin{definition}\label{porous}
		A set $P$ in a metric space $(X,d)$ is porous if there exists $0<\rho<1$ such that for all $x\in P$ and $\delta>0$, there exists $y\in X$ with $d(y,x)<\delta$ such that
		\[B(y,\rho d(y,x))\cap P=\varnothing.\]
		A set is called $\sigma$-porous if it is a countable union of porous sets.
	\end{definition}
	
	Clearly porosity is sensitive to the choice of metric. Unless otherwise stated, we will use Euclidean distance on $I$ and the metric $d$ on $F$.
	
	Recall that a subset of a metric space is of first category or meager if it is a countable union of nowhere dense sets. A property of points in a metric space holds for typical points if the set where it does not hold is of first category. Clearly every porous set is nowhere dense and every $\sigma$-porous set is of first category. In the case of a metric measure space $(X,d,\mu)$ equipped with a doubling measure $\mu$, including the Laakso space $F$, porous sets have measure zero. This is well known. For an explicit proof one could follow the steps in \cite{PS17}, which do not rely on the Carnot group structure in that paper.
	
	\section{Maximal Directional Derivatives and Differentiability}\label{Maximaldirectional}
	
	In this section we classify geometrically the set of points where maximality of a directional derivative implies differentiability. We then show this set is $\sigma$-porous, so intuitively the set is very small.
	
	\subsection{Directional Derivatives}
	
	We first show that differentiability is a stronger requirement than directional differentiability as one would expect.
	
	\begin{lemma}\label{lemma_deriv_implies_diff}
		If a function $f\colon F\to \bbR$ is differentiable at a point $x\in F$ with derivative $Df(x)$, then $f_{I}(x)$ exists and equals $Df(x)$.
		
		For any $x\in F$, there exists a Lipschitz function $f\colon F\to \bbR$ such that $f_{I}(x)$ exists but $f$ is not differentiable at $x$.
	\end{lemma}
	
	\begin{proof}
		For the first part, first suppose that $f$ is differentiable at $x = [x_1,x_2]\in F$ with derivative $Df(x)$ and assume that $x$ is not a wormhole. Then we have
		\begin{align*}
			0 &=\lim_{t\to 0}\frac{f[x_1+t,x_2]-f(x)-Df(x)(h[x_1+t,x_2]-h(x))}{d([x_1+t,x_2],x)}\\
			& = \lim_{t\to 0}\frac{f[x_1+t,x_2]-f(x) - t\cdot Df(x)}{t}\\
			& =\lim_{t\to 0} \frac{f[x_1+t,x_2]-f(x)}{t}-Df(x).
		\end{align*}
		Hence $f_{I}(x)$ exists and equals $Df(x)$. The case in which $x$ is a wormhole is done by computing $f_{L}$ and $f_{R}$ separately.
		
		For the second part, first fix $x\in F$ and assume $x_{1}\notin \{0, 1\}$. Let $V$ denote the line through $x$ in the $I$ direction. Recall the definitions of $D_{n}^{+}$ and $D_{n}^{-}$ from \eqref{dplus} and \eqref{dminus}. Fix $N\in \bbN$ sufficiently large that $D_{n}^{+}(x_{1}),D_{n}^{-}(x_{1})$ are finite for all $n\geq N$. If $x$ is a wormhole, so that $x_{1}\in J_{M}$ for some $M$, then we additionally choose $N$ such that $N>M$.   Next, for each $n\geq N$ we define the following points:
		\begin{itemize}
			\item $u_{n}$ is the point vertically above $x$ at a vertical distance $D_{n}^{+}(x_{1})$, 
			\item $d_{n}$ is the point vertically below $x$ at a vertical distance $D_{n}^{-}(x_{1})$,
			\item $y_{n}$ is the point obtained by starting at $x$, travelling up to $u_{n}$, using the wormhole to jump to the identified point, then travelling back down to the point with the same height as $x$.
		\end{itemize}
		Let $A=V\cup \{y_{n}: n\geq N\}$ and define $f\colon A\to \bbR$ by
		\[f|_{V}=0 \qquad \mbox{and} \qquad f(y_{n})=\min(D_{n}^{+}(x_{1}), D_{n}^{-}(x_{1})) \mbox{ for }n\geq N.\]
		Clearly the directional derivative $f_{I}(x)$ exists and equals $0$.
		
		Notice $d(x,y_{n})=\min(2D_{n}^{+}(x_{1}),2D_{n}^{-}(x_{1}))$ for all $n\geq N$. This gives
		\[\frac{f(y_{n})-f(x)}{d(y_{n},x)}=\frac{\min(D_{n}^{+}(x_{1}), D_{n}^{-}(x_{1}))}{\min(2D_{n}^{+}(x_{1}), 2D_{n}^{-}(x_{1}))}=\frac{1}{2}\not\to 0.\]
		Hence $f$ is not differentiable at $x$.
		
		To see that $f$ is Lipschitz it will suffice to estimate the values of  $|f(y_n)-f(u_n)|$, $|f(y_n)-f(d_n)|$, and $|f(y_n)-f(y_m)|$ for $n, m\geq N$. First notice that for $n\geq N$.
		\[|f(y_n)-f(u_n)|=\min(D_{n}^{+}(x_{1}), D_{n}^{-}(x_{1}))\leq D_{n}^{+}(x_1)=d(y_n, u_n)\]
		and
		\[|f(y_n)-f(d_n)|=\min(D_{n}^{+}(x_{1}), D_{n}^{-}(x_{1}))\leq D_{n}^{-}(x_1)=d(y_n, d_n).\]
		Now suppose $n, m\geq N$. Notice that we can choose a geodesic from $y_{n}$ to $y_{m}$ which passes either through $u_{n}$ or through $d_{n}$. Suppose one passes through $u_{n}$. Then
		\begin{align*}
			|f(y_{n})-f(y_{m})|&\leq |f(y_{n})-f(u_{n})|+|f(u_{n})-f(y_{m})|\\
			&\leq d(y_{n},u_{n})+d(u_{n},y_{m}).\\
			&=d(y_{n},y_{m}).
		\end{align*}
		A similar argument applies if the geodesic passes through $d_{n}$ rather than $u_{n}$. This shows that $f\colon A\to \bbR$ is $1$-Lipschitz. Extending $f$ to a Lipschitz function on $F$ proves the second part of the lemma in the case $x_{1}\notin \{0,1\}$. The proof is similar if $x_{1}=0$ or $x_{1}=1$, with adjustments to make the construction one-sided.
	\end{proof}
	
	We next show the Lipschitz constant can be recovered as the supremum of directional derivatives. This justifies our definition of maximal directional derivative.
	
	\begin{proposition}\label{Lipissup}
		Let $f\colon F\to \bbR$ be Lipschitz. Then
		\[\mathrm{Lip}(f)=\sup \left\{ |f_{I}(x)|: x\in F \mbox{ and }f_{I}(x) \mbox{ exists} \right\}.\]
	\end{proposition}
	
	\begin{proof}
		Temporarily define $\mathrm{Lip}_{D}(f)=\sup \{ |f_{I}(x)|: x\in F \mbox{ and }f_{I}(x) \mbox{ exists} \} $. Fix $x=[x_1,x_2]\in F$ such that $f_I(x)$ exists. If $x$ is not a wormhole, then
		\begin{align*}
			|f_I(x)| &= \left|\lim_{t\to 0}\frac{f[x_1 + t,x_2] - f(x)}{t}\right| \\ 
			&\leq \limsup_{t\to 0}\frac{\mathrm{Lip}(f) \cdot d([x_1+t,x_2],x)}{t}\\
			&= \mathrm{Lip}(f).
		\end{align*}
		The same argument applies when $x$ is a wormhole, since $f_I(x) = f_L(x) = f_R(x)$. This proves that $ \mathrm{Lip}_{D}(f) \leq \mathrm{Lip}(f)$.
		
		Next fix $x, y\in F$ and let $L=d(x,y)$. Choose a geodesic $\gamma\colon [0,L]\to F$ from $x$ to $y$ which is a concatenation of countably many lines in the $I$-direction. More precisely, there is a decomposition $[0,L]=\cup_{i=1}^{\infty}I_{i}$ so that for $i\geq 1$:
		\begin{itemize}
			\item $I_{i}\subset [0,L]$ are closed intervals overlapping only pairwise at endpoints.
			\item There are $a_{i}\in [0,1]$ and $x_{i}\in K$ so that $\gamma|_{I_{i}}(t)=[a_{i}\pm t, x_{i}]$.
		\end{itemize}
		Since $f\circ \gamma$ is absolutely continuous, we can estimate as follows:
		\begin{align*}
			|f(y)-f(x)|&=|f(\gamma(L))-f(\gamma(0))|\\
			&=\left|\int_{0}^{L} (f\circ \gamma)'(s) \dd s \right|\\
			&\leq L \sup \{|(f\circ \gamma)'(s)|: s\in (0,L) \mbox{ and }(f\circ \gamma)'(s)\mbox{ exists}\}\\
			&\leq L \mathrm{Lip}_{D}(f).
		\end{align*}
		Since $L=d(x,y)$, this gives $\mathrm{Lip}(f)\leq \mathrm{Lip}_{D}(f)$ and completes the proof.
	\end{proof}
	
	\subsection{Relationship Between $M$ and $S$}
	
	Recall the set $M$ from Definition \ref{defM} and the set $S$ from Definition \ref{defS}. To begin studying the relationship between them, we start with the following simple lemma.
	
	\begin{lemma}\label{analysis}
		The following statements hold.
		\begin{enumerate}
			\item For any $t\in (0,1)$, we have for all sufficiently large $n$
			\[D_{n}^{+}(t) \leq 2/3^{n} \qquad \mbox{and} \qquad D_{n}^{-}(t)\leq 2/3^{n}.\]
			\item For any $t\in S$, we have for all sufficiently large $n$
			\[D_{n}^{+}(t) \geq c(t)/3^{n} \qquad \mbox{and} \qquad D_{n}^{-}(t) \geq c(t) /3^{n},\]
			where $c(t)=1/(1+C(t))$.
			\item For any $t\in (0,1)\setminus S$, we have
			\[ \limsup_{n\to \infty} \frac{D_{n}^{+}(t)}{D_{n}^{-}(t)}=\infty \qquad \mbox{or} \qquad \limsup_{n\to \infty} \frac{D_{n}^{-}(t)}{D_{n}^{+}(t)}=\infty.\]
		\end{enumerate}
	\end{lemma}
	
	\begin{proof}
		Statement (1) holds because adjacent elements of $J_{n}$ are separated by at most a distance $2/3^{n}$ away from heights $0$ and $1$; the factor $2$ is necessary because of the requirement $m_{n}>0$ in the definition of $J_{n}$. 
		
		Statement (2) follows from the estimate
		\[1/3^n \leq D_{n}^{+}(t)+D_{n}^{-}(t)\leq (1+C(t))D_{n}^{+}(t),\]
		which yields $D_{n}^{+}(t) \geq c(t)/3^{n}$ with $c(t)=1/(1+C(t))$. A similar argument yields the estimate for $D_{n}^{-}(t)$. 
		
		Statement (3) follows from negating the definition of $S$.
	\end{proof}
	
	
	\begin{proposition}\label{diff}
		Suppose $x=[x_{1},x_{2}] \in F$ with $x_{1}\in  S$. Then $x\in M$. In other words, every Lipschitz map $f\colon F\to \bbR$ with a maximal directional derivative at the point $x$ is also differentiable at $x$.
	\end{proposition}
	
	\begin{proof} Fix $f\colon F\to \bbR$ Lipschitz with $|f_{I}(x)|=\mathrm{Lip}(f)$. Without loss of generality we assume $f_{I}(x)=\mathrm{Lip}(f)$, otherwise replace $f$ by $-f$. Let $L:=\mathrm{Lip}(f)$. We show
		\begin{equation}\label{lim_differentiability}
			\lim_{y\to x} \frac{f(y)-f(x)-L(h(y)-h(x))}{d(y,x)}=0.
		\end{equation}
		Fix $N(x_{1})\in \bbN$, $C(x_{1})\geq 1$, and $0<c(x_{1})<1$ such that \eqref{fractionpm} and Lemma \ref{analysis}(2) hold with $t=x_{1}$ for all $n\geq N(x_{1})$.
		
		\emph{Case 1:} Suppose $x$ is not a wormhole. Fix $\varepsilon>0$. Let $y\in F$ be sufficiently close to $x$ in a sense to be made precise below. Assume $y_{1}\geq x_{1}$; the case $y_{1}<x_{1}$ is similar. Let $N$ be the minimal $n\in \bbN$ for which every path connecting $x$ and $y$ must pass through a point of $F$ whose height belongs to $J_{n}$. By making $y$ sufficiently close to $x$, we may assume that $N\geq N(x_{1})$. Since every path connecting $x$ and $y$ must pass through a height in $J_{N}$, it follows that
		\begin{equation}\label{dlower}
			d(x,y)\geq \min\{ D_{N}^{+}(x_{1}), D_{N}^{-}(x_{1})\}\geq c(x_{1})/3^{N},
		\end{equation}
		where the fact $N\geq N(x_{1})$ and Lemma \ref{analysis}(2) was used for the second inequality. 
		
		Let $z_{N}:=[x_{1}-2/3^{N}, x_{2}]$. Using $f_{I}(x)=L$ and assuming $y$ is sufficiently close to $x$, which makes $N$ sufficiently large, we can ensure
		\[f(z_{N})-f(x)\leq L(h(z_{N})-h(x))+(\varepsilon/3^{N}).\]
		Notice that
		\[h(y)-h(z_{N})\geq h(x)-h(z_{N})=2/3^{N}\]
		and in any interval of length $2/3^{N}$ we can find elements of $J_{n}$ for all $n\geq N$. Hence $d(y,z_{N})=h(y)-h(z_{N})$. Using \eqref{dlower} for the final line, we now estimate as follows
		\begin{align*}
			f(y)-f(x)&= (f(y)-f(z_{N}))+(f(z_{N})-f(x))\\
			& \leq Ld(y,z_{N}) +L(h(z_{N})-h(x))+(\varepsilon/3^{N}) \\
			&= L(h(y)-h(z_{N})) +L(h(z_{N})-h(x))+(\varepsilon/3^{N})\\
			&= L(h(y)-h(x))+(\varepsilon/c(x_{1})) d(x,y).
		\end{align*}
		
		For the opposite inequality, let $w_{N}:=[y_{1}+2/3^{N},x_{2}]$. Provided $y$ is sufficiently close to $x$, which ensures $y_{1}$ is close to $x_{1}$ and $N$ is sufficiently large, we may use $f_{I}(x)=L$ to obtain
		\[f(w_{N})-f(x)\geq L(h(w_{N})-h(x))-\varepsilon (h(w_{N})-h(x)).\]
		Since any interval of length $2/3^N$ contains elements of $J_{n}$ for all $n\geq N$, we have $d(y,w_{N})=h(w_{N})-h(y)$. Using \eqref{dlower} for the final line,  we estimate as follows
		\begin{align*}
			f(y)-f(x) &= (f(y)-f(w_{N}))+(f(w_{N})-f(x))\\
			&\geq -Ld(y,w_{N}) + L(h(w_{N})-h(x))-\varepsilon (h(w_{N})-h(x))\\
			&=L(h(y)-h(w_{N})+h(w_{N})-h(x))-\varepsilon ((h(y)-h(x)+(2/3^{N}))\\
			&\geq L(h(y)-h(x))-\varepsilon d(x,y)(1+2/c(x_{1})).
		\end{align*}
		Hence, for $y$ sufficiently close to $x$,
		\[
		-\varepsilon(1+2/c(x_1))\leq\frac{f(x)-f(y)-L(h(x)-h(y))}{d(x,y)}\leq \varepsilon/c(x_1)
		\]
		which proves the limit \eqref{lim_differentiability}.
		
		\emph{Case 2:} Suppose $x$ is a wormhole of order $K\geq 1$, so $x_{1}\in J_{K}$. The argument is similar to that of Case 1, but we give the details for completeness. Let $x_{2}$ be the smaller of the two elements of $K$ satisfying $x=[x_{1},x_{2}]$. Fix $\varepsilon>0$. Let $y\in F$ be sufficiently close to $x$ in a sense to be made precise below. Assume $y_{1}\geq x_{1}$. The case $y_{1}<x_{1}$ is similar. Let $N$ be the minimal $n\in \mathbb{N}$ for which every path connecting $x$ and $y$ must pass through a height in $J_{n}\setminus J_{K}$. By making $y$ sufficiently close to $x$, we may assume that $N\geq N(x_{1})$. Since every path connecting $x$ and $y$ must pass through a height in $J_{N}$, it follows as before that
		\[d(x,y)\geq \min\{ D_{N}^{+}(x_{1}), D_{N}^{-}(x_{1})\}\geq c(x_{1})/3^{N}.\]
		
		Let $z_{N}=[x_{1}-2/3^{N}, x_{2}]$. Using $f_{I}(x)=L$ and assuming $y$ is sufficiently close to $x$, we have
		\[f(z_{N})-f(x)\leq L(h(z_{N})-h(x))+(\varepsilon/3^{N}).\]
		Notice that
		\[h(y)-h(z_{N})\geq h(x)-h(z_{N})=2/3^{N}\]
		and in any interval of length $2/3^{N}$ we can find elements of $J_{n}$ for all $n\geq N$. Hence $d(y,z_{N})=h(y)-h(z_{N})$, so we can estimate as follows
		\begin{align*}
			f(y)-f(x)&= (f(y)-f(z_{N}))+(f(z_{N})-f(x))\\
			& \leq Ld(y,z_{N})+L(h(z_{N})-h(x))+(\varepsilon/3^{N}) \\
			&= L(h(y)-h(z_{N}))+L(h(z_{N})-h(x))+(\varepsilon/3^{N})\\
			&= L(h(y)-h(x))+(\varepsilon/c(x_{1})) d(x,y).
		\end{align*}
		
		For the opposite inequality, fix $p\subset I\times K$ such that $q(p)$ is a path joining $x$ with $y$ and $\mathcal{H}^{1}(p)=d(x,y)$. If $(x_{1},x_{2})\in p$ we let $w_{N}:=[y_{1}+2/3^{N},x_{2}]$. Otherwise $p$ contains $(x_{1}, x_{2}+2/3^{K})$ and we let $w_{N}:=[y_{1}+2/3^{N},x_{2}+2/3^{K}]$. With this choice of $w_{N}$, the points $y$ and $w_{N}$ are separated only by jumps of level $n\geq N$. Assuming $y$ is sufficiently close to $x$, we may use $f_{I}(x)=L$ to obtain 
		\[f(w_{N})-f(x)\geq L(h(w_{N})-h(x))-\varepsilon (h(w_{N})-h(x)).\]
		Since any interval of length $2/3^N$ contains elements of $J_{n}$ for all $n\geq N$, we have $d(y,w_{N})=h(w_{N})-h(y)$. Using \eqref{dlower} for the final line,  we estimate as follows
		\begin{align*}
			f(y)-f(x) &= (f(y)-f(w_{N}))+(f(w_{N})-f(x))\\
			&\geq -Ld(y,w_{N}) + L(h(w_{N})-h(x))-\varepsilon (h(w_{N})-h(x))\\
			&=L(h(y)-h(w_{N})+h(w_{N})-h(x))-\varepsilon (h(y)-h(x)+(2/3^{N}))\\
			&\geq L(h(y)-h(x))-\varepsilon d(x,y)(1+(2/c(x_{1}))).
		\end{align*}
		Hence for $y$ sufficiently close to $x$,
		\[
		-\varepsilon(1+2/c(x_1))\leq\frac{f(x)-f(y)-L(h(x)-h(y))}{d(x,y)}\leq \varepsilon/c(x_1)
		\]
		which proves the limit \eqref{lim_differentiability} also for the case when $x$ is a wormhole and $y_1\geq x_1$.
	\end{proof}
	
	%
	%
		
	\begin{proposition}\label{nondiff}
		Suppose $x=[x_{1},x_{2}] \in F$ with $x_{1} \notin S$. Then $x\notin M$. In other words, there exists a Lipschitz map $f\colon F\to \bbR$ with a maximal directional derivative at $x$ which is not differentiable at $x$.
	\end{proposition}
	
	\begin{proof}
	We give the proof in the case $x_{1}\in (0,1)\setminus S$. If $x_{1} \in \{0,1\}$ the proof would be largely the same, except the construction is made only on one side of $x$; vertically above if $x_{1}=0$ and vertically below if $x_{1}=1$.

		For simplicity let $D_{n}^{+}:= D_{n}^{+}(x_{1})$ and $D_{n}^{-}:=D_{n}^{-}(x_{1})$. It follows that either
		\[ \limsup_{n\to \infty} D_{n}^{-}/D_{n}^{+} =\infty \qquad \mbox{or} \qquad \limsup_{n\to \infty}D_{n}^{+}/ D_{n}^{-}=\infty\]
		We assume that $\limsup_{n\to \infty} D_{n}^{-}/D_{n}^{+} =\infty$; the argument in the other case is similar with the construction inverted in the $I$ direction. Choose a strictly increasing sequence $n_{k}$ such that $D_{n_{k}}^{-}, D_{n_{k}}^{+}$ are finite for all $k\geq 1$ and $D_{n_{k}}^{-}/D_{n_{k}}^{+}\to \infty$. Since $D_{n}^{+}\to 0$ and $D_{n}^{-}\to 0$ as $n\to \infty$, by taking a subsequence if necessary, we may assume $D_{n_{k}}^{-}$ and $D_{n_{k}}^{+}$ are each strictly decreasing and for every $k\geq 1$,
		\[ 2D_{n_{k+1}}^{-}/D_{n_{k}}^{-}<1/n_{k} \qquad \mbox{and} \qquad 2D_{n_{k}}^{+}/D_{n_{k}}^{-}< 1/n_{k}.\]
		Fix a sequence $0<\theta_{n_{k}}<1$ which satisfies $\theta_{n_{k}}\to 1$ as $k\to \infty$ and
		\begin{equation}\label{thetadef}
			\theta_{n_{k}}\leq \left( 1 - \frac{D_{n_{k}}^{+}+D_{n_{k+1}}^{-}}{D_{n_{k}}^{-}} \right) / \left( 1 - \frac{D_{n_{k+1}}^{-}}{D_{n_{k}}^{-}} \right) \qquad \mbox{for all }k\geq 1.
		\end{equation}
		Note that, since the right hand side converges to $1$ as $k\to \infty$, these conditions on $\theta_{n_{k}}$ can be realized. Let $J:=[x_{1}-D_{n_{1}}^{-},x_{1}+D_{n_{1}}^{+}]\subset \bbR$. Define $\varphi\colon J\to [0,1]$ by 
		\[\varphi(t)=1 \qquad \mbox{if } t\geq x_{1}\]
		and
		\[ \varphi(t)=\theta_{n_{k}} \qquad \mbox{if } x_{1}-D_{n_{k}}^{-}\leq t<x_{1}-D_{n_{k+1}}^{-} \mbox{ for some }k\geq 1.\]
		
		\bigskip
		
		Since $x_{1}\notin S$, we know that $x$ is not a wormhole level. For all $k\geq 1$, let $y_{n_{k}}\in F$ be the endpoint of the path which starts at $x$, travels vertically up along the line segment to height $x_{1}+D_{n_{k}}^{+}$, jumps using the level $J_{n_{k}}$, then travels vertically down along the line segment to height $x_{1}$. Thus $y_{n_{k}}=[x_{1},x_{2}\pm 2/3^{n_{k}}]$ where the choice of sign may depend on $k$. Since $D_{n_{k}}^{+}<D_{n_{k}}^{-}$, we have $d(y_{n_{k}},x)=2D_{n_{k}}^{+}$. Now let
		\[A:=\{ [t,x_{2}]\in F \colon t\in J\} \cup \{y_{n_{k}}: k\in \bbN\}.\]
		Since $x$ is not a wormhole level, the sets $\{ [t,x_{2}]\in F \colon t\in J\}$ and $\{y_{n_{k}}: k\in \bbN\}$ are disjoint. Hence we may define $f\colon A\to \bbR$ by 
		\[f[t,x_{2}]= \int_{x_{1}}^{t} \varphi(s)\dd s \qquad \mbox{for }t\in J\]
		and
		\[f(y_{n_{k}})=D_{n_{k}}^{+} \qquad \mbox{for }k\in \bbN.\]
		Notice that $f(x)=0$.
		
		\bigskip
		
		\emph{Claim.} $f$ is $1$-Lipschitz with respect to the restriction of $d$ to $A$.
		
		\begin{proof}[Proof of Claim]
			Suppose $a=[t,x_{2}]$ and $b=[s,x_{2}]$ for some $t,s\in J$. Then $|\varphi|\leq 1$ implies
			\[|f(a)-f(b)|\leq |t-s|=d(a,b).\]
			Hence $f$ is $1$-Lipschitz restricted to the set $\{ [t,x_{2}]\in F \colon t\in J\}$.
			
			Next let $u_{n_{k}}$ be the point reached by starting at $x$ and travelling vertically up along the line segment to height $x_{1}+D_{n_{k}}^{+}$. Similarly let $d_{n_{k}}$ be the point reached by starting at $x$ and travelling vertically down along the line segment to height $x_{1}-D_{n_{k}}^{-}$. It follows from the definition of $f$ that 
			\[f(u_{n_{k}})=D_{n_{k}}^{+}=f(y_{n_{k}}).\]
			On the other hand we have, using the definitions of $f$ and $\varphi$,
			\begin{align*}
				f(y_{n_{k}})-f(d_{n_{k}}) &= f(u_{n_{k}})-f(d_{n_{k}})\\
				&= \int_{x_{1}-D_{n_{k}}^{-}}^{x_{1}+D_{n_{k}}^{+}}\varphi (s) \dd s\\
				&= \theta_{n_{k}}(D_{n_{k}}^{-}-D_{n_{k+1}}^{-}) + \int_{x_{1}-D_{n_{k+1}}^{-}}^{x_{1}}\varphi(s)\dd s + D_{n_{k}}^{+}.
			\end{align*}	
			Using $|\varphi|\leq 1$, $d(d_{n_{k}},y_{n_{k}})=D_{n_{k}}^{-}$, and the choice of $\theta_{n_{k}}$ in \eqref{thetadef}, we obtain
			\begin{align*}
				\frac{ |f(y_{n_{k}})-f(d_{n_{k}})| }{ d(d_{n_{k}},y_{n_{k}}) } &\leq \theta_{n_{k}}\left( 1 - \frac{D_{n_{k+1}}^{-}}{D_{n_{k}}^{-}} \right) + \frac{D_{n_{k+1}}^{-}}{D_{n_{k}}^{-}}+\frac{D_{n_{k}}^{+}}{D_{n_{k}}^{-}}\\
				&\leq 1.
			\end{align*}
			
			Suppose $a=[t,x_{2}]$ for some $t\in J$ and $k\geq 1$. Every geodesic from $y_{n_{k}}$ to $a$ must pass through either $u_{n_{k}}$ or $d_{n_{k}}$. Denote such a point by $z_{n_{k}}$; the argument will be the same in either case. Then we have
			\[d(y_{n_{k}},a)=d(y_{n_{k}},z_{n_{k}})+d(z_{n_{k}},a).\]
			Using also what was proved above, we have
			\begin{align*}
				|f(y_{n_{k}})-f(a)| &\leq |f(y_{n_{k}})-f(z_{n_{k}})| + |f(z_{n_{k}})-f(a)|\\
				&\leq d(y_{n_{k}},z_{n_{k}}) + d(z_{n_{k}}, a)\\
				&=d(y_{n_{k}},a).
			\end{align*}
			
			It remains to estimate $|f(y_{n_{k}})-f(y_{n_{l}})|$ for $k>l\geq 1$. Define points $u_{n_{l}}, u_{n_{k}}$ and $d_{n_{l}}, d_{n_{k}}$ as before. A geodesic from $y_{n_{l}}$ to $y_{n_{k}}$ is obtained by the following curve:
			\begin{enumerate}
				\item Start at $y_{n_{l}}$,
				\item Travel vertically upward to the wormhole $u_{n_{l}}$,
				\item Jump using wormhole $u_{n_{l}}$ and travel downwards to the wormhole $u_{n_{k}}$,
				\item Jump using wormhole $u_{n_{k}}$ and travel downwards to the point $y_{n_{k}}$.
			\end{enumerate}
			Thus $d(y_{n_{l}},y_{n_{k}})=2D_{n_{l}}^{+}$ for $k>l$. Hence we can estimate
			\begin{align*}
				|f(y_{n_{l}})-f(y_{n_{k}})|&=|D_{n_{l}}^{+}-D_{n_{k}}^{+}|\\
				&=D_{n_{l}}^{+}-D_{n_{k}}^{+}\\
				&\leq d(y_{n_{l}},y_{n_{k}}).
			\end{align*}
			This concludes the proof  of the claim.
		\end{proof}
		
		Now extend $f$ arbitrarily to a $1$-Lipschitz function $f\colon F\to \bbR$.
		
		\bigskip
		
		\emph{Claim.} The directional derivative $f_{I}(x)$ exists and equals $1$.
		
		\begin{proof}[Proof of Claim]
			If $t\in [-D_{n_{1}}^{-}, D_{n_{1}}^{+}]$ then 
			\[\frac{f[x_{1}+t,x_{2}] - f[x_{1},x_{2}]}{t} = 1+ \frac{1}{t}\int_{x_{1}}^{x_{1}+t} (\varphi(s)-1)\dd s.\]
			Since $\varphi(s)=1$ for $s\geq x_{1}$ we obtain
			\[\frac{f[x_{1}+t,x_{2}] - f[x_{1},x_{2}]}{t} =1 \qquad \mbox{for all }t\geq 0.\]
			Fix $\varepsilon>0$ and fix $K$ large enough so that $|\theta_{n_{k}}-1|<\varepsilon$ for all $k\geq K$. Then for $0>t>-D_{n_{K}}^{-}$ we obtain
			\begin{align*}
				\left| \frac{1}{t}\int_{x_{1}}^{x_{1}+t} (\varphi(s)-1)\dd s \right| &\leq \varepsilon.
			\end{align*}
			This proves the claim.
		\end{proof}
		
		\emph{Claim.} $f$ is not differentiable at $x$.
		
		\begin{proof}[Proof of Claim]
			Recall that $f(x)=0$, $f(y_{n_{k}})=D_{n_{k}}^{+}$ and $d(y_{n_{k}},x)=2D_{n_{k}}^{+}$. Hence for any $\lambda \in \bbR$ we have,
			\begin{align*}
				\frac{f(y_{n_{k}})-f(x)-\lambda (h(y_{n_{k}})-h(x))}{d(y_{n_{k}},x)} &= \frac{f(y_{n_{k}})-f(x)}{d(y_{n_{k}},x)} \\
				&=\frac{1}{2}.\\
			\end{align*}
			Since $y_{n_{k}}\to x$ as $k\to \infty$, this shows that $f$ is not differentiable at $x$. 
		\end{proof}
		This proves the proposition.
	\end{proof}

	\subsection{Porosity}
	
	We have now shown that $M=\{[x_{1},x_{2}]\colon x_{1}\in S\}$. We now study the set $S$ and show that it is $\sigma$-porous, hence a relatively small set.
	
	\begin{lemma}\label{preimage}
		If $A\subset I$ is porous (respectively $\sigma$-porous), then $h^{-1}(A)\subset F$ is also porous (respectively $\sigma$-porous).
	\end{lemma}
	
	\begin{proof}
		Since $A$ is porous in $I$, there exists $C>0$ such that for every $t\in A$ there is a sequence $t_{n}\in I$ with $t_{n}\to t$ such that
		\begin{equation}\label{porousI}
			B(t_{n},C|t_{n} - t|)\cap A=\varnothing \qquad \mbox{for every }n\in \bbN.
		\end{equation}
		Fix $[t,x]\in h^{-1}(A)$. Then $t\in A$. Hence there exists a sequence $t_{n}\in I$ with $t_{n}\to t$ such that \eqref{porousI} holds. Consider the sequence $[t_{n},x]\in F$. Clearly $[t_{n},x]\to [t,x]$ with respect to the natural metric on $F$. Let 
		\[B_{n}=B([t_{n},x],Cd([t_{n},x],[t,x])).\] 
		We claim that
		\begin{equation}\label{porousF}
			B_{n} \cap h^{-1}(A)=\varnothing. \qquad \mbox{for every }n\in \bbN
		\end{equation}
		To this end, fix $n\in \bbN$ and suppose $[s,y]\in B_{n}$. Then
		\begin{align*}
			|s-t_{n}|&\leq d([s,y],[t_{n},x])\\
			&\leq Cd([t_{n},x],[t,x])\\
			&=C|t_{n}-t|.
		\end{align*}
		Hence $s\in B(t_{n}, C|t_{n}-t|)$. By \eqref{porousI}, this implies $s\notin A$. Hence $[s,y]\notin h^{-1}(A)$. This shows \eqref{porousF}, so $h^{-1}(A)$ is porous in $F$ as required.
	\end{proof}
	
	We can now prove Proposition \ref{SMsmall}.
	
	\begin{proposition}\label{SMsmall}
		The set $S\subset I$ is $\sigma$-porous in $I$, hence first category and of Lebesgue measure zero. 
		
		The set $M\subset F$ is $\sigma$-porous in $F$, hence first category and of $\mathcal{H}^{Q}$ measure zero.
	\end{proposition}
	
	\begin{proof}
		We can write
		\[S=\bigcup_{\substack{C\in \mathbb{Q}\\C>1}} \bigcup_{N\in \bbN} S_{C,N}\]
		where
		\[S_{C,N}=\left\{t\in (0,1)\colon D_{n}^{-}(t), D_{n}^{+}(t)<\infty \mbox{ and }C^{-1}\leq \frac{D_{n}^{+}(t)}{D_{n}^{-}(t)} \leq C \mbox{ for all } n\geq N\right\}.\]
		Fix $C\in \mathbb{Q}$ with $C>1$ and $N\in \bbN$. We show that the set $S_{C,N}\subset (0,1)$ is porous. Fix $0<\lambda<1/2$ such that $(1-\lambda)/\lambda > C$. Let $t\in J_{n}$ for some $n\geq N$. We will show that:
		\begin{equation}\label{empty}
			\left(t, t+\frac{\lambda}{3^n} \right) \cap S_{C,N}=\varnothing.
		\end{equation} 
		First fix $s\in (t,t+\lambda/3^n)$. Then $D_{n}^{-}(s)\leq \lambda/3^n$ and $D_{n}^{+}(s)\geq 1/3^{n}-\lambda/3^n$. Combining these inequalities gives
		\[\frac{D_{n}^{+}(s)}{D_{n}^{-}(s)}\geq \frac{1-\lambda}{\lambda} >C.\]
		Hence $s\notin S_{C,N}$, which establishes  \eqref{empty}.
		
		Next fix $t_{0}\in S_{C,N}$ and $\delta>0$. Choose $n>N$ with $2/3^n<\delta$ and $t\in J_{n}$ with $|t-t_{0}|<2/3^{n}$. Then $(t,t+\lambda/3^{n})\cap S_{C,N}=\varnothing$. This shows that $S_{C,N}$ is porous and hence $S$ is $\sigma$-porous. 
		
		The second part of the proposition follows by Lemma \ref{preimage} and $M=h^{-1}(S)$. Finally, for the implication in each case, we recall porous sets are nowhere dense and have measure zero with respect to doubling measures. Hence $\sigma$-porous sets are first category and measure zero with respect to doubling measures.
	\end{proof}
	
	\begin{theorem}[Restatement of Theorem \ref{mainthmmaxdiff}]
		Let $M$ be the set of $x\in F$ such that whenever $f\colon F\to \bbR$ is Lipschitz with directional derivative $f_{I}(x)=\pm \lip(f)$ at $x$, then $f$ is differentiable at $x$. 
		
		Then $M=\{[x_{1},x_{2}]\colon x_{1}\in S\}$. In particular $M$ is $\sigma$-porous.
	\end{theorem}
	
	\begin{proof}
		The theorem follows from combining Proposition \ref{diff}, Proposition \ref{nondiff}, and Proposition \ref{SMsmall}.
	\end{proof}

	\section{Differentiability of the Distance Function}\label{distdiff}

In this section we study where the distance to a fixed point is differentiable. We show the set where the distance to a fixed point is not differentiable is $\sigma$-porous.

\subsection{Analysis of Distance}

We begin by proving some simple properties of the distance map $y\mapsto d_{p}(y):=d(y,p)$ for each fixed point $p\in F$.

\begin{lemma}\label{claim1}
	Fix $p=[p_1,p_2]\in F$. Suppose $x=[x_1,x_2]\in F\setminus\{p\}$ is not a wormhole. Then there exists $0<\Delta<1$ such that if $|t|<\Delta$ and $y_{2}\in K$ with $|y_{2}-x_{2}|<\Delta$, then
	\[d_{p}[x_{1}+t,x_{2}]=d_{p}[x_{1}+t,y_{2}].\]
\end{lemma}

\begin{proof}
	We divide into cases depending on whether a wormhole is needed to connect $p$ to $x$. Suppose no wormhole level is needed to connect $p$ to $x$, namely $x_{2}=p_{2}$. Choose $0<\Delta<\frac{1}{2}|x_{1}-p_{1}|$ sufficiently small that if $y_{2}\in K$ with $|y_{2}-x_{2}|<\Delta$, then the wormhole levels required to join $p_{2}$ to $y_{2}$ can be found in $(p_{1},p_{1}+\frac{1}{2}(x_{1}-p_{1}))$ if $x_{1}>p_{1}$ or in $(p_{1}-\frac{1}{2}(p_{1}-x_{1}),p_{1})$ if $x_{1}<p_{1}$. Then for every $t\in (-\Delta,\Delta)$, 
	\[d_{p}[x_{1}+t,x_{2}]=d_{p}[x_{1}+t,y_{2}]=|(x_{1}+t)-p_{1}|.\]
	This proves the lemma in the case no wormholes are needed to connect $p$ to $x$.
	
	Now suppose wormholes are needed to connect $p$ to $x$, namely $x_{2}\neq p_{2}$. Define
	\[N:=\min \{n\in \bbN: \mbox{a wormhole of level $n$ is required to join $p$ to $x$}\}.\]
	Choose $\Delta > 0$ sufficiently small that for all $y_{2}\in K$ with $|y_{2}-x_{2}|<\Delta$:
	\begin{enumerate}
		\item If $p_{1}<1$ and $D_{N}^{+}(p_{1})<\infty$, every wormhole level needed to join $x_{2}$ to $y_{2}$ can be found at heights in $(p_{1},p_{1}+D_{N}^{+}(p_{1}))$.
		\item If $p_{1}>0$ and $D_{N}^{-}(p_{1})<\infty$, every wormhole level needed to join $x_{2}$ to $y_{2}$ can be founded at heights in $(p_{1}-D_{N}^{-}(p_{1}),p_{1})$.
		\item A wormhole level in $J_{N}$ is required to connect $p_{2}$ to $y_{2}$.
	\end{enumerate}
	Now fix $t\in (-\Delta, \Delta)$ and $y_{2}\in K$ with $|y_{2}-x_{2}|<\Delta$. 
	
	Let $\gamma$ be a geodesic from $p$ to $[x_{1}+t,x_{2}]$. Using the definition of $N$, either
	\begin{itemize}
		\item[(a)] $p_{1}<1$, $D_{N}^{+}(p_{1})<\infty$, $\gamma$ passes through all heights in $(p_{1},p_{1}+D_{N}^{+}(p_{1}))$, or
		\item[(b)] $p_{1}>0$, $D_{N}^{-}(p_{1})<\infty$, $\gamma$ passes through all heights in $(p_{1}-D_{N}^{-}(p_{1}),p_{1})$.
	\end{itemize}
	Using (1) and (2), we may modify $\gamma$ without changing its length to obtain a curve $\widetilde{\gamma}$ connecting $p$ to $[x_{1}+t,y_{2}]$. This gives $d_{p}[x_{1}+t,y_{2}] \leq d_{p}[x_{1}+t,x_{2}]$.
	
	Conversely, let $\eta$ be a geodesic from $p$ to $[x_{1}+t,y_{2}]$. Using (3), it follows that (a) or (b) hold again. Using (1) and (2), we can modify $\eta$ without changing its length to obtain a curve $\widetilde{\eta}$ connecting $p$ to $[x_{1}+t,x_{2}]$. This gives $d_{p}[x_{1}+t,y_{2}] \geq d_{p}[x_{1}+t,x_{2}]$.
	
	Combining the two inequalities concludes the proof.
\end{proof}

The proof of the following lemma is similar to that of Lemma \ref{claim1}.

\begin{lemma}\label{claim2}
	Fix $p\in F$. Suppose $x\in F\setminus\{p\}$ is a wormhole. Fix $x_{1}\in I$ and $x_{2}, x_{2}'\in K$ with $x_{2}<x_{2}'$ such that  $x=[x_{1},x_{2}]=[x_{1},x_{2}']$. Then there exists $0<\Delta<1$ such that, for $|t|<\Delta$:
	\begin{itemize}
		\item If $y_{2}\in K$ with $|y_{2}-x_{2}|<\Delta$, then $d_{p}[x_{1}+t,x_{2}]=d_{p}[x_{1}+t,y_{2}]$.
		\item If $y_{2}\in K$ with $|y_{2}-x_{2}'|<\Delta$, then $d_{p}[x_{1}+t,x_{2}']=d_{p}[x_{1}+t,y_{2}]$.
	\end{itemize}
\end{lemma}

\begin{proposition}\label{teo_derivimpliesdiff}
	Fix $p\in F$. Then the map $d_{p}\colon F\to \bbR$ is differentiable at a point $x\in F\setminus\{p\}$ if and only if the directional derivative $(d_{p})_{I}(x)$ exists.
\end{proposition}

\begin{proof}
	Clearly if $d_{p}$ is differentiable at a point $x\in F$, then the directional derivative $(d_{p})_{I}(x)$ exists. We show the converse. Suppose $x\in F$ and $D:=(d_{p})_{I}(x)$ exists. 
	
	Assume $x$ is not a wormhole and let $\varepsilon>0$. Using the definition of the directional derivative, we can find $\delta>0$ such that whenever $t\in I$ with $0<|t|<\delta$ we have
	\begin{equation*}\label{direcdiffJan28}
		|d_{p}[x_{1}+t,x_{2}]-d_{p}[x_{1},x_{2}]-tD|<\varepsilon|t|.
	\end{equation*}
	Fix $\Delta >0$ as in Lemma \ref{claim1}. If $0<|t|<\min(\delta,\Delta)$ and $y_{2}\in K$ with $|y_{2}-x_{2}|<\Delta$,
	\[|d_{p}[x_{1}+t,y_{2}]-d_{p}[x_{1},x_{2}]-tD|<\varepsilon|t|.\]
	Every $y\in F$ sufficiently close to $x$ can be represented as $y=[x_{1}+t,y_{2}]$ where $0<|t|<\min(\delta,\Delta)$ and $y_{2}\in K$ with $|y_{2}-x_{2}|<\Delta$. In addition $t=h(y)-h(x)$ and $|t|\leq d(x,y)$. Hence for all $y$ close enough to $x$ we have
	\begin{equation}\label{diffjan28}
		|d_{p}(y)-d_{p}(x)-D(h(y)-h(x))|<\varepsilon d(x,y).
	\end{equation}
	Hence $d_{p}$ is differentiable at $x$ with derivative $D = (d_p)_I(x)$.
	
	If $x$ is a wormhole the proof is similar, using Lemma \ref{claim2} instead.
\end{proof}

Due to Proposition \ref{teo_derivimpliesdiff} we can reduce the study of the differentiability of $d_p(\cdot)$ to the study of the directional derivative. Before doing so we give two definitions.

\begin{definition}\label{def_vertical_line}
For each point $p=[p_{1},p_{2}]\in F$, we define the sets of points which can be reached from $p$ using different numbers of jumps as follows.
\begin{itemize}
\item We define $V_{0}^{p}$ to be the set of points which can be reached starting at $p$ with no additional jumps. If $p$ is not a wormhole then $V_{0}^{p}$ is a single  line in the $I$ direction through $p$. If $p$ is a wormhole then $V_{0}^{p}$ consists of two lines.

\item If $N$ is a positive integer and $p$ is not a wormhole of level $N$, we define $V_{N}^{p}$ to be the set of all points in $F$ that can be reached from $p$ by jumping exactly once through a wormhole of level $N$. More explicitly, if $p=[p_1,p_2]\in F$ is not a wormhole then
	\[
	V_N^p = \{[t,p_2\pm 2/3^N]\,:\,t\in[0,1] \},
	\]
	where the choice of $+$ or $-$ is uniquely determined by $p_2$ and $N$. Similarly if $p=[p_1,p_2]=[p_1,p'_2]$ with $p_{2}\neq p_{2}'$ is a wormhole of level $M\neq N$, then $V_{N}^{p}$ is composed of two vertical lines:
	\[V_N^p = \{[t,p_2\pm 2/3^N]: t\in[0,1] \} \cup \{[t,p_{2}'\pm 2/3^N]: t\in[0,1] \},\]
	where the sign is determined by $p$ and $N$.

	\item More generally, let $N_1<N_2<\ldots$ be positive integers and assume $p$ is not a wormhole of level $N_{i}$ for any $i$. We define $V_{\Delta}^p$ to be the collection of points in $F$ which can be reached from $p$ by jumping exactly once in each wormhole of level $N_1,N_2,\cdots$. More explicitly, let $
	\Delta = \pm\frac{2}{3^{N_1}}\pm\frac{2}{3^{N_2}}\pm\ldots$
	where the signs are uniquely determined by $p$ and $N_{1}, N_{2}, \cdots$. If $p$ is not a wormhole, then
	\[
	V_{\Delta}^p = \left\{ \left[t,p_2+\Delta\right]: t\in[0,1] \right\}.
	\]
	Similarly if $p=[p_1,p_2]=[p_1,p'_2]$ with $p_{2}\neq p_{2}'$ is a wormhole, then
	\[V_{\Delta}^p = \left\{ \left[t,p_2+\Delta\right]: t\in[0,1] \right\}\cup \left\{ \left[t,p_{2}'+\Delta\right]: t\in[0,1] \right\}.\]
\end{itemize}
\end{definition}

\begin{definition}\label{def_up_ending}
	An \emph{upward going segment} in $F$ is a curve $\gamma:[a,b]\to F$ of the form $\gamma(t)=[\lambda +t, \mu]$ for some $\lambda\in[0,1]$, $\mu\in K$. Similarly a \emph{downward going segment} is a curve $\gamma:[a,b]\to F$ of the form $\gamma(t)=[\lambda -t, \mu]$ for some $\lambda\in[0,1]$, $\mu\in K$.
	
A curve $\gamma\colon [a,b]\to F$ is \emph{upward (downward) ending} if there exists $\delta>0$ such that the restriction of $\gamma$ to $[b-\delta, b]$ is an upward (downward) going segment.
\end{definition}

We divide the study of the directional derivative of $d_p(\cdot)$ into four steps, depending on how many jumps are required to join $p$ to the point under consideration. The following proposition is immediate.


\begin{proposition}\label{vp0}
Let $p\in F$. Then there is only one point in $V_{0}^{p}$ where the directional derivative of $d_p(\cdot)$  does not exist, namely the point $p$ itself.
\end{proposition}

We now study differentiability at points reached from $p$ by exactly one jump. First we make the following observation.


\begin{lemma}\label{lemma_double_geod}
	Let $p\in F$. Suppose for a point $q\neq p\in F$ there exists both a downward ending geodesic and an upward ending geodesic from $p$ to $q$. Then $d_p$ is not differentiable at $q$. 
\end{lemma}

\begin{proof}
	Let $p$ and $q=[q_1,q_2]$ be as in the hypothesis. Let $\gamma^u:[a,b]\to F$ be the upward ending geodesic and $\gamma^d:[a,b]\to F$ the downward ending one (we can take the same starting interval up to parametrizations). From Definition \ref{def_up_ending} there exists $\delta>0$ such that the restriction of both geodesics to the interval $[b-\delta,b]$ is respectively an upward ending or downward ending segment. Take $-\delta<t<\delta$. Then the point $[q_1+t,q_2]$ belongs to either $\gamma^{u}([b-\delta,b])$ or $\gamma^{d}([b-\delta,b])$ for $t<0$ or $t>0$ respectively. To define a geodesic that connects $p$ to $[q_1+t,q_2]$ we can use the restriction of $\gamma^u$ or $\gamma^d$ to the interval $[a,b-|t|]$. By construction, the length of this geodesic is the length of $\gamma^u$ (or $\gamma^d$) minus $|t|$. Hence $d_p([q_1+t,q_2])-d_p(q)=-|t|$ and the limit \eqref{limit_vert_deriv} does not exist.
\end{proof}

\begin{proposition}\label{prop_vertical_1}
	Let $p\in F$ and fix an integer $N>0$, different from $M$ if $p$ is a wormhole of level $M$. Then there are only a finite number of points in $V^p_N$ where the directional derivative of $d_p(\cdot)$ does not exist.
\end{proposition}

\begin{proof}
	Let $p=[p_1,p_2]$ and $N$ be as in the hypothesis. We prove the proposition in the case $p$ is not a wormhole. If $p$ is a wormhole then the proof is similar except the argument is repeated twice, once for each vertical line in $V_N^p$. We split the argument into several cases.
	
\begin{claim}\label{claima}
Suppose $D_N^+(p_1)$ exists but $D_N^-(p_1)$ does not. Then there is only one point in $V_N^p$ at which $(d_p)_I$ does not exist. This point is $x=[p_1+D_N^+(p_1),p_2]$.

A similar statement holds if  $D_N^-(p_1)$ exists and $D_N^+(p_1)$ does not. Then the point of non-differentiability is $x=[p_1-D_N^-(p_1),p_2]$.
\end{claim}
	
\begin{proof}
Suppose that only $D_N^+(p_1)$ exists; the other case is similar. First we show that $d_p(\cdot)$ is not differentiable at the point $x$ defined above. A point $y\in V_N^p$ close to $x$ can be written as $y=[p_1+D_N^+(p_1)+t,p_2']$, where $p_2'=p_2\pm\frac{2}{3^N}$ with the choice of $+$ or $-$ uniquely determined by $p_2$. Note that $t$ can take both positive and negative values. To go from $p$ to $y$ with a geodesic we must start from $p$, go up to height $p_1+D_N^+(p_1)$, jump to the line $V_N^p$ through the wormhole $x$ and then go up or down by height $t$ to the point $y$. This has the effect of adding a segment of length $|t|$ to the original geodesic that connected $p$ to $x$. Since $D_N^-(p_1)$ does not exist, this new path is clearly a geodesic that connects $p$ to $y$ and $d_p(y)-d_p(x)=|t|$. Hence the limit \eqref{limit_vert_deriv} for the function $d_p(\cdot)$ does not exists at $x$. 

We are left to show that $(d_p)_I$ exists at any other point $y\in V_N^p$. Take a point $y=[y_1,p_2']\in V_N^p$ with $y_1>p_1+D_N^+(p_1)$. For $t$ such that $p_1+D_N^+(p_1)<y_1+t\leq 1$ we have $d_p([y_1+t,p_2])-d_p(y)=t$ so the limit \eqref{limit_vert_deriv} exists and is equal to 1. Similarly, if $y_1<p_1+D_N^+(p_1)$, we see that  $d_p([y_1+t,p_2])-d_p(y)=-t$ for $t$ sufficiently close to $0$ so the limit \eqref{limit_vert_deriv} exists and is equal to $-1$. Hence $d_p(\cdot)$ is differentiable at any point $y\neq x\in V_N^p$.
\end{proof}
	
	\smallbreak

	\begin{claim}\label{claimb}
	Suppose that both $D_N^+(p_1)$ and $D_N^-(p_1)$ exist. In this case there are exactly three points in $V_{N}^{p}$ at which $(d_p)_I$ does not exist: 
	\begin{itemize}
	\item $x_1:=[p_1+D_N^+(p_1),p_2']$,
	\item $x_2:=[p_1-D_N^-(p_1),p_2']$,
	\item $x_3:=[p_1+c_N,p_2']$, where $c_N:= D_N^+(p_1)-D_N^-(p_1)$.
	\end{itemize}
		\end{claim}
	
\begin{proof}
Note  $h(x_1)>h(x_3)>h(x_2)$. The proof that $(d_p)_I$ does not exist at the points $x_{1}, x_{2}$ is similar to the proof of Claim \ref{claima}. We show that $(d_p)_I$ does not exist at $x_{3}$.  To prove this we consider two paths $\gamma_{1}, \gamma_{2}$ from $p$ to $x_{3}$.

The first path $\gamma_{1}$ starts at $p$, goes up to height $p_1+D_N^+(p_1)$, jumps in the wormhole of level $N$ $x_1$, then goes down to height $p_1+c_N$. The length of $\gamma_1$ is $D_{N}^{+}(p_{1})+(D_{N}^{+}(p_{1})-c_{N})=D_{N}^{+}(p_{1})+D_{N}^{-}(p_{1})$. The second path $\gamma_{2}$ starts at $p$, goes down to height $p_1-D_N^-(p_1)$, jumps in the wormhole of level $N$ $x_2$, then goes up to height $p_1+c_N$. The length of $\gamma_2$ is $D_{N}^{-}(p_{1})+(c_{N}-D_{N}^{-}(p_{1}))=D_{N}^{-}(p_{1})+D_{N}^{+}(p_{1})$. Since both $\gamma_{1}$ and $\gamma_{2}$ have the same length and any geodesic from $p$ to $x_{3}$ must go through at least the same heights as for $\gamma_{1}$ or for $\gamma_{2}$, it follows $\gamma_{1}$ and $\gamma_{2}$ are both geodesics from $p$ to $x_{2}$. Since $\gamma_{1}$ ends downwards and $\gamma_{2}$ ends upwards, it follows by Lemma \ref{lemma_double_geod} that $d_{p}$ is not differentiable at $x_{3}$.

	
	It remains to prove that $d_p(\cdot)$ is differentiable at every point of $V_{N}^{p}$ distinct from $x_1,x_2$ and $x_3$. If $y=[y_1,p_2']\in V_N^p$ with $y_1>p_1+D_N^+(p_1)$ or $y_1<p_1-D_N^-(p_1)$, we can use the same argument as in the previous claim. The remaining points are of the type $[y_1,p_2']$ with $p_1-D_N^-(p_1)<y_1<p_1+D_N^+(p_1)$ and $y_1\neq p_1+C_N$. Take $y_1$ such that $p_1+c_N<y_1<p_1+D_N^+(p_1)$ (the other case works in the same way) and consider $t$ such that $p_1+c_N<y_1+t<p_1+D_N^+(p_1)$. Since the points $[y_1,p_2']$ and $[y_1+t,p_2']$ are both above $x_3$, a geodesic from $p$ to each point can be obtained by shortening the same downward ending geodesic that connects $p$ to $x_3$ from the previous step. Hence
\[d_p([y_1,p_2'])=d_p(x_3)-(y_1-p_1-c_N)\]
and
\[d_p([y_1+t,p_2'])=d_p(x_3)-(y_1+t-p_1-c_N).\]
Hence $d_p([y_1+t,p_2'])-d_p([y_1,p_2'])=-t$, i.e. the limit \eqref{limit_vert_deriv} exists and $d_p(\cdot)$ is differentiable at $y$.
	\end{proof}


This concludes the proof of the proposition.
\end{proof}

Now we can show what happens on vertical lines that two jumps away from $p$.

\begin{proposition}\label{prop_vertical_2}
	Let $p\in F$ and fix integers $0<N<M$, both different from $W$ if $p$ is a wormhole of level $W$. Then there exists a finite number of points in $V^p_{\Delta}$, where $\Delta = \pm\frac{2}{3^{N}}\pm\frac{2}{3^{M}}$ and the signs are uniquely determined by $p_2$, in which the vertical derivative of $d_p(\cdot)$ does not exist. 
\end{proposition}

\begin{proof}

We prove the proposition under the assumption that $p=[p_{1},p_{2}]$ is not a wormhole. When $p$ is a wormhole we get twice as many points of non-differentiability and a remark similar to that for Proposition \ref{prop_vertical_1} applies. Since we will only consider jump levels relative to the point $p$, we use the simpler notation $D_N^{\pm}, D_M^{\pm}$ for $D_N^{\pm}(p_1), D_M^{\pm}(p_1)$.  There are several cases depending on which of $D_N^{\pm}, D_M^{\pm}$ exist and their relative positions. We first study the case when all four exist. 

\smallskip

\textbf{Suppose all four of $D_N^{\pm}, D_M^{\pm}$ exist.} Then we will have three possible cases. Note that the case $-D_M^-<-D_N^-<D_N^+<D_M^+$ cannot occur. Indeed, between the height of any two wormholes of level $N$ we can find the height of at least one wormhole of level $M$ for any $M>N$. 

\smallskip
	
\textbf{Case (a):} $-D_N^-<-D_M^-<D_M^+<D_N^+$. 

Given any $x\in V_{\Delta}^p$ consider a geodesic that connects $p$ to $x$. Such a geodesic must jump through a wormhole of level $N$ and a wormhole of level $M$. However, since we are in case $(a)$, between $p_1$ and the height of the first available wormhole of level $N$ there is always the height of a wormhole of level $M$. Suppose we follow the same geodesic that connects $p$ to $x$, except we jump only in the wormhole of level $N$ and not the wormhole of level $M$. We arrive at a point $x'$ with the following properties:
\begin{enumerate}
	\item $x'\in V_N^p$,
	\item $h(x)=h(x')$,
	\item $d_p(x)=d_p(x')$.
\end{enumerate}
Hence in $V_{\Delta}^p$ there are the same number of points at which $d_p$ is not differentiable as there are in $V_N^p$, hence finitely many.

\smallskip

\textbf{Case (b):} $-D_N^-<-D_M^-<D_N^+<D_M^+$.

Let $p_2'=p_2+\Delta$, $c_N := D_N^{+}-D_N^-$, and $c_M := D_M^{+}-D_M^-$. We will show that the following are all the points of non-differentiability for $d_{p}$ on $V_{\Delta}^{p}$:
\begin{equation}\label{x_even}
\begin{split}
&x_1=[p_1-D_N^-,p_2'],\ x_3=[p_1-D_M^-,p_2'],\\
&x_5=[p_1+D_N^+,p_2'],\ x_7=[p_1+D_M^+,p_2'],
\end{split}
\end{equation}
and
\begin{equation}\label{x_odd}
x_2 := [p_1+c_N,p_2'], \ x_4:=[p_1,p_2'], \ x_6:=[p_1+c_M,p_2'].
\end{equation}
Roughly speaking, $x_{1}, x_{3}, x_{5}, x_{7}$ are points where geodesics split, while $x_{2}, x_{4}, x_{6}$ are points which can be reached by both an upwards ending and downwards ending geodesic.

%

\begin{claim}\label{enumeration_points}
The sequence $h(x_{i})$ is increasing for $1\leq i\leq 7$.
\end{claim}

\begin{proof}
The inequality $h(x_1)<h(x_3)<h(x_5)<h(x_7)$ follows directly from the definition of the points and the hypotheses of Case (b). Similarly the inequality  $h(x_3)<h(x_4)<h(x_5)$ follows immediately from the definition of $x_4$. We are left to prove $h(x_1)<h(x_2)<h(x_3)$ and $h(x_5)<h(x_6)<h(x_7)$. Clearly
\[h(x_1) = p_1-D_N^{-} < p_1-D_N^{-} + D_N^+ = h(x_2)\]
and
\[h(x_6) = p_1+D_M^{+} -D_M^- < p_1+D_M^{+} = h(x_7).\]

To prove $h(x_2)<h(x_3)$ we need to show
\begin{equation}\label{Apr28.1}
D_{N}^{+}-D_{N}^{-}<-D_{M}^{-}.
\end{equation}
Note $D_{N}^{+}+D_{M}^{-}=1/3^M$ since it represent the distance from a wormhole of level $N$ to one of the nearest wormhole of level $M$. Moreover, by construction, $D_N^->1/3^M$, since the change in heights $D_{N}^{-}-D_{M}^{-}$ is equal to $1/3^{M}$. Hence $D_{N}^{+}+D_{M}^{-}<D_{N}^{-}$ as desired. This proves $h(x_2)<h(x_3)$.

To prove $h(x_{5})<h(x_{6})$ it suffices to show 
\begin{equation}\label{Apr28.2}
	D_{N}^{+}<D_{M}^{+}-D_{M}^{-},
\end{equation}
i.e. $D_{N}^{+}+D_{M}^{-}<D_{M}^{+}$. To see this note once again that $D_{N}^{+}+D_{M}^{-}=1/3^{M}$ and $D_{M}^{+}>D_{M}^{+}-D_{N}^{+}=1/3^{M}$, from which we conclude $D_{N}^{+}+D_{M}^{-}<D_M^+$ as desired.
\end{proof}

\begin{claim}\label{claim_even_points} 
$d_{p}$ is not differentiable at $x_{1}, x_{3}, x_{5}, x_{7}$.
\end{claim}

\begin{proof}
We first prove $d_{p}$ is not differentiable at $x_{1}$. We will show that there exists $\delta>0$ such that 
\[d_p([p_1-D_N^{-}+t,p_2'])-d_p(x_1)=|t| \mbox{ for }-\delta<t<\delta.\]
	
	For $0<t<p_1-D_N^-$, a minimal height interval for $p$ and $[p_1-D_N^{-} -t,p_2']$ is the interval $[p_1-D_N^{-} -t,p_1]$. Hence $d_p([p_1-D_N^{-} -t,p_2'])=D_N^{-}+t$ and since $d_p(x_1)=D_N^-$, we conclude that $d_p([p_1-D_N^{-}-t,p_2'])-d_p(x_1)=t$. Similarly, for $0<t<D_N^+$ a minimal height interval for the points $p$ and $[p_1-D_N^{-}+t,p_2']$ is still the interval $[p_1-D_N^-,p_1]$, whose length is $D_N^-$. Indeed if this were not the case then the only other possible choice would be the interval $[p_1-D_N^{-}+t,p_1+D_N^+]$ (we recall that a minimal height interval must contain the height of all the wormhole needed to connect the two points) whose length is $D_N^{-}+D_N^{+}-t$. This, together with our initial choice $0<t<D_N^+$, contradicts minimality since $D_N^{-}+D_N^{+}-t>D_N^-$. Hence the interval $[p_1-D_N^-,p_1]$ is still the minimal height interval for $p$ and $[p_1-D_N^{-}+t,p_2']$. If we use this to compute the distance we get $d_p([p_1-D_N^{-}+t,p_2'])=D_N^{-}+t$ from which we conclude that also in this case $d_p([p_1-D_N^{-}+t,p_2'])-d_p(x_1)=t$. We now define $\delta:=\min\{p_1-D_N^-,D_N^+\}$. What we proved so far is that, for $-\delta<t<\delta$, $d_p([p_1-D_N^{-}+t,p_2'])-d_p(x_1)=|t|$, hence the limit \eqref{limit_vert_deriv} does not exist and $d_p(\cdot)$ is not differentiable at $x_1$.
	
	For the point $x_{7}$ the proof is similar to that of $x_{1}$ with minimal height interval $[p_1,p_1+D_M^+]$.
	
	We now show $d_{p}$ is not differentiable at $x_{3}$. We claim the unique minimal height interval for a geodesic from $p$ to $x_{3}$ is $[p_1-D_M^-,p_1+D_N^+]$. Comparing the length of the geodesic from this minimal height interval with its competitors, it suffices to check
\begin{equation}\label{MinimalintervalApril25}
D_{N}^{+}+D_{N}^{+}+D_{M}^{-}<D_{N}^{-}+D_{N}^{-}-D_{M}^{-}.
\end{equation}
Rearranging and simplifying, this is equivalent to $D_{N}^{+}+D_{M}^{-}<D_{N}^{-}$ which is exactly \eqref{Apr28.1}. From this it follows that any geodesic from $p$ to $x_{3}$ starts at $p$, moves up to height $p_{1}+D_{N}^{+}$, jumps through the wormhole of level $N$, moves down to height $p_{1}-D_{M}^{-}$, then jumps through the wormhole of level $M$. Any geodesic from $p$ to a point $[p_1-D_M^{-}+t,p_2']$ near $x_{3}$ for sufficiently small $t$ follows the same path followed by a small movement in the $I$ direction. Hence $d_{p}([p_1-D_M^{-}+t,p_2'])-d_{p}(x_{3})=|t|$, leading to non-differentiability at $x_{3}$.

For the point $x_{5}$ the proof is similar to that of $x_{3}$ and the minimal height interval is again $[p_1-D_M^{-},p_1+D_N^{+}]$.
\end{proof}

\begin{claim}\label{claim_odd_points} 
$d_{p}$ is not differentiable at $x_{2}, x_{4}, x_{6}$.
\end{claim}

\begin{proof}
 We show $x_2,x_4$, $x_6$ satisfy the hypotheses of Lemma~\ref{lemma_double_geod}. We begin with $x_2$. Note $D_{N}^{+}-D_{N}^{-}<-D_{M}^{-}$ from \eqref{Apr28.1}. Hence $h(x_{2})<p_{1}-D_{M}^{-}$. To get a downward ending path $\gamma_d$ (respectively upwards path $\gamma_u$), connecting $p$ to $x_2$, we proceed as follows:
\begin{enumerate}
	\item start from $p$ and we go up to height $p_1+D_N^+$ (respectively down to $p_1-D_M^-$),
	\item jump with the wormhole of level $N$ (respectively $M$) found there,
	\item go down to height $p_1-D_M^-$ (respectively down to $p_1-D_N^-$),
	\item jump with the wormhole of level $M$ (respectively $N$) found there,
	\item go down (respectively up) to height $p_1+c_N$.
\end{enumerate}
To see the length of $\gamma_{u}$ and $\gamma_{d}$ are the same it suffices to check
\[D_{N}^{+}+D_{N}^{+}+D_{M}^{-}+(-D_{M}^{-}-c_{N})=D_{M}^{-}+(-D_{M}^{-}+D_{N}^{-})+(c_{N}+D_{N}^{-})\]
or equivalently
\[2D_{N}^{+}-c_{N}=2D_{N}^{-}+c_{N}\]
which follows from the definition of $c_{N}$. Any geodesic connecting $p$ to $x_{2}$ must pass through either all the heights in $\gamma_{u}$ or all the heights in $\gamma_{d}$. Hence $\gamma_{u}$ and $\gamma_{d}$ are both geodesics. Since $\gamma_{d}$ ends downwards and $\gamma_{u}$ ends upwards, it follows that $d_{p}$ is non differentiable at $x_{2}$.

For connecting $p$ to $x_4$ consider the downward (respectively upward) ending paths obtained as follows:
\begin{enumerate}
	\item From $p$ we go up to height $p_1+D_N^+$ (respectively down to $p_1-D_M^-$);
	\item we jump with the wormhole of level $N$ (respectively $M$) found there;
	\item we go down to height $p_1-D_M^-$ (respectively up to  $p_1+D_N^+$);
	\item we jump with the wormhole of level $M$ (respectively $N$) we found there;
	\item we go up (respectively down) to height $p_1$.
\end{enumerate}
Both paths have equal length $2D_{N}^{+}+2D_{M}^{-}$. To see they are both geodesics, it suffices to see $2D_{N}^{+}+2D_{M}^{-}<\min (2D_{M}^{+}, 2D_{N}^{-})$. However this follows from \eqref{Apr28.1} and \eqref{Apr28.2}. Since one path ends upwards and one path ends downward, non differentiability at $x_{4}$ then follows from Lemma \ref{lemma_double_geod}.


The construction of the geodesics from $p$ to $x_6$ is omitted for brevity as it is similar to the construction of the geodesics from $p$ to $x_2$.
\end{proof}

\begin{claim}
The points $x_1,\cdots, x_7$ are the only points in $V_{\Delta}^p$ at which $d_p$ is not differentiable.
\end{claim}

\begin{proof}
If $y=[y_1,p_2']\in V_{\Delta}^p$ with $y_1<h(x_1)$ or $y_1>h(x_7)$, we can use the same argument used in the end of the proof of Claim \ref{claima} to show that $(d_p)_I$ exists at $y$.

The remaining points are of the type $y=[y_1,p_2']$ with $h(x_i)<y_1<h(x_{i+1})$ for $i=1\leq i\leq 6$.

\bigskip

Assume $h(x_{1})<y_{1}<h(x_{2})$. We claim that a geodesic from $p$ to $y$ starts at $p$, goes down to height $p_{1}-D_{N}^{-}$, then goes up to height $y_{1}$ with jumps where necessary. The length of such a curve is $2D_{N}^{-}+y_{1}-p_{1}$. To see it is a geodesic we must compare with the competitor which starts at $p$, goes up to height $p_{1}+D_{N}^{+}$, goes down to height $p_{1}-D_{M}^{-}$, then down to height $y_{1}$ with jumps in between. The length of such a curve is $2D_{N}^{+}+p_{1}-y_{1}$. Hence we need 
\[2D_{N}^{-}+y_{1}-p_{1} < 2D_{N}^{+}+p_{1}-y_{1}\]
which rearranges to
\[y_{1}<D_{N}^{+}-D_{N}^{-}+p_{1}.\]
This is guaranteed by the assumption $y_{1}<h(x_{2})$. Hence $d_{p}(y)=2D_{N}^{-}+y_{1}-p_{1}$ for all $h(x_{1})<y_{1}<h(x_{2})$. It follows that $(d_{p})_{I}(y)=1$ in this case.

\bigskip

Assume $h(x_{2})<y_{1}<h(x_{3})$. In this case a geodesic starts at $p$, goes up to height $p_{1}+D_{N}^{+}$, goes down to height $p_{1}-D_{M}^{-}$, then goes down to height $y_{1}$ with necessary jumps in between. The natural competitor is the geodesic in the previous case and the argument is the same with inequalities reversed. It follows that $(d_{p})_{I}(y)=-1$ in this case.

\bigskip

Assume $h(x_{3})<y_{1}<h(x_{4})$. We claim the natural geodesic starts at $p$, goes up to height $p_{1}+D_{N}^{+}$, goes down to height $p_{1}-D_{M}^{-}$, then up to height $y_{1}$ with jumps in between. The length of such a geodesic is $2D_{N}^{+}+2D_{M}^{-}+y_{1}-p_{1}$ so this would imply $(d_{p})_{I}(y)=1$. That the proposed curve is already shorter than the one which starts at $p$ and goes down to height $p_{1}-D_{N}^{-}$ is contained in the previous cases. We compare it with the curve which starts at $p$, goes up to height $p_{1}+D_{M}^{+}$, then goes down to height $y_{1}$ with necessary jumps in between. The length of such a curve is $2D_{M}^{+}+p_{1}-y_{1}$. To see that the claimed geodesic is shorter we need
\[2D_{N}^{+}+2D_{M}^{-}+y_{1}-p_{1} < 2D_{M}^{+}+p_{1}-y_{1}\]
which rearranges to
\[D_{N}^{+}+D_{M}^{-}-D_{M}^{+}<p_{1}-y_{1}.\]
However we already know $D_{N}^{+}<D_{M}^{+}-D_{M}^{-}$ so the required inequality is true as long as $y_{1}<p_{1}$ which is guaranteed by the assumption $y_{1}<h(x_{4})$ in this case. A similar argument applies in comparison to the curve which starts at $p$, goes down to height $p_{1}-D_{M}^{-}$, goes up to height $p_{1}+D_{N}^{+}$, then goes down to height $y_{1}$.

\bigskip

Assume $h(x_{4})<y_{1}<h(x_{5})$. We claim the geodesic in this case starts at $p$, goes down to height $p_{1}-D_{M}^{-}$, goes up to height $p_{1}+D_{N}^{+}$, then down to height $y_{1}$. The length of such a curve is $2D_{M}^{-}+2D_{N}^{+}+p_{1}-y_{1}$. This can be compared to the curve which starts at $p$, goes up to height $p_{1}+D_{N}^{+}$, goes down to height $p_{1}-D_{M}^{-}$, then goes up to height $y_{1}$. This has length $2D_{N}^{+}+2D_{M}^{-}+y_{1}-p_{1}$ which is clearly longer. We can also compare to the curve which starts at $p$, goes up to height $p_{1}+D_{M}^{+}$, then goes down to height $y_{1}$. This has length $2D_{M}^{+}+p_{1}-y_{1}$. To see this is shorter we need
\[2D_{M}^{-}+2D_{N}^{+}+p_{1}-y_{1} < 2D_{M}^{+}+p_{1}-y_{1}.\]
Rearranging gives $D_{M}^{-}+D_{N}^{+}<D_{M}^{+}$ which was also established earlier. Hence $(d_{p})_{I}(y)=-1$ in this case.

\bigskip

Assume $h(x_{5})<y_{1}<h(x_{6})$. We claim that the geodesic in this case starts at $p$, goes down to height $p_{1}-D_{M}^{-}$, then goes up to height $y_{1}$ with jumps in between. The length of such a curve is $2D_{M}^{-}+y_{1}-p_{1}$. This should be compared against the curve which starts at $p$, goes up to height $p_{1}+D_{M}^{+}$, then down to height $y_{1}$ with jumps in between. Such a curve has height $2D_{M}^{+}+p_{1}-y_{1}$. It suffices to see
\[2D_{M}^{-}+y_{1}-p_{1} < 2D_{M}^{+}+p_{1}-y_{1}\]
which rearranges to
\[y_{1}<p_{1}+D_{M}^{+}-D_{M}^{-}.\]
This in turn is guaranteed by the assumption $y_{1}<h(x_{6})$. Hence $d_{p}(y)=2D_{M}^{-}+y_{1}-p_{1}$ and so $(d_{p})_{I}(y)=1$ in this case.

\bigskip

Assume $h(x_{6})<y_{1}<h(x_{7})$. In this case the geodesic starts at $p$, goes up to height $p_{1}+D_{M}^{+}$, then down to height $y_{1}$ with jumps in between. Such a curve has height $2D_{M}^{+}+p_{1}-y_{1}$. That this is the shortest curve follows by a similar argument to the previous case with signs reversed. Hence $(d_{p})_{I}(y)=-1$ in this case.

\end{proof}

This concludes the proof of Case (b).

\smallskip

\textbf{Case (c):} $-D_M^-<-D_N^-<D_M^+<D_N^+$.

 This case works the same as in Case $(b)$ with upwards and downwards directions reversed and we find the same number of points. 
 

\smallskip

\textbf{Suppose one or more of $D_N^{\pm}, D_M^{\pm}$ do not exist.} Suppose $p_{1}\leq 1/2$, the case $p_{1}>1/2$ is similar with up and down orientations reversed. Then $D_{N}^{+}$ and $D_{M}^{+}$ both exist. The cases are whether $D_{M}^{-}$ does not exist, $D_{N}^{-}$ does not exist, or both do not exist. Note that the case when only $D_{M}^{-}$ does not exist cannot occur, because between height $0$ and height $p_{1}-D_{N}^{-}$ there would exist a wormhole of level $M$ because $0<N<M$.

\smallskip

\textbf{Case (a)}: Suppose neither $D_{M}^{-}$ and $D_{N}^{-}$ exist. Then it is easy to see that there is one point of non-differentiability of $d_{p}$ on the line $V_{\Delta}^{p}$. If $D_{N}^{+}<D_{M}^{+}$ then the point is $[p_{1}+D_{M}^{+},p_{2}+\Delta]$, while if $D_{M}^{+}<D_{N}^{+}$ then the point is $[p_{1}+D_{N}^{+},p_{2}+\Delta]$.

\smallskip

\textbf{Case (b)}: Suppose $D_{M}^{-}$ exists but $D_{N}^{-}$ does not exist. We split into two cases depending on whether $D_{N}^{+}$ or $D_{M}^{+}$ is larger.

Suppose $D_{N}^{+}>D_{M}^{+}$. In this case there is only one point of non-differentiability at the point $[p_{1}+D_{N}^{+},p_{2}+\Delta]$. This is because in order to jump across to the line $V^p_{\Delta}$ from $p$ every geodesic must go up from $p$ to height $p_{1}+D_{N}^{+}$.

Suppose $D_{M}^{+}>D_{N}^{+}$. In this case there are five points of non-differentiability:
\[x_{1}=[p_{1}-D_{M}^{-},p_{2}'], \qquad x_{3}=[p_{1}+D_{N}^{+},p_{2}'], \qquad x_{5}=[p_{1}+D_{M}^{+},p_{2}']\]
\[ x_{2}=[p_{1},p_{2}'], \qquad x_{4}=[p_{1}+D_{M}^{+}-D_{M}^{-},p_{2}'].\]
The points $x_{1}, x_{2}, x_{3}, x_{4}, x_{5}$ are in order of increasing height. The points $x_{1}, x_{3}, x_{5}$ are wormhole levels where geodesics split. The points $x_{2}, x_{4}$ are points where up and down ending geodesics meet. The proof that $x_{1}, x_{2}, x_{3}, x_{4}, x_{5}$ are exactly the points of non-differentiability is analogous to the previous steps in the proof.

This concludes the proof of Proposition \ref{prop_vertical_2}.

\end{proof}

Finally we count the points of non differentiability on lines $V_{\Delta}^p$ where $\Delta$ comes from any sequence $N_1<N_2<\cdots$ of three or more wormhole levels. However, the following lemma tells us that these points have already been accounted for.

\begin{lemma}\label{lemma_parallel_values}
		Let $p\in F$ and fix integers $0<N_1<N_2<\cdots$, each different from $M$ if $p$ is a wormhole of level $M$. 
		
		Let $p=[p_{1},p_{2}]$, with the understanding there are two possible choices of $p_{2}$ if $p$ is a wormhole and the following holds for both choices. Let $\Delta'=\pm\frac{2}{3^{N_1}}\pm\frac{2}{3^{N_2}}$ and $\Delta = \sum_i \pm\frac{2}{3^{N_i}}$, where the signs are uniquely determined by $p$. Then for each $x=[t, p_{2}+\Delta] \in V_{\Delta}^p$, the point $\overline{x}=[t,p_{2}+\Delta'] \in V_{\Delta'}^p$ satisfies $d_{p}(\overline{x})=d_{p}(x)$.
		
Consequently, for each $x=[t, p_{2}+\Delta] \in V_{\Delta}^p$ the distance $d_{p}$ is differentiable at $x$ if and only if it is differentiable at $\overline{x}=[t,p_{2}+\Delta'] \in V_{\Delta'}^p$.
\end{lemma}

\begin{proof}
	Let $p$ and the integers $N_1,N_2,\dots$ be as in the hypotheses. Fix any point $x=[t, p_{2}+\Delta] \in V_{\Delta}^p$ with corresponding $\overline{x}=[t,p_{2}+\Delta'] \in V_{\Delta'}^p$. 
	
	Suppose $\gamma$ is a geodesic from $p$ to $x$. Then we can form a curve $\overline{\gamma}$ from $p$ to $\overline{x}$ simply by modifying $\gamma$ so as not to jump through any wormhole level other than $N_{1}$ and $N_{2}$. This does not increase the length of $\gamma$, hence $d_{p}(\overline{x})\leq d_{p}(x)$.
	
	Conversely, suppose $\overline{\gamma}$ is a geodesic from $p$ to $\overline{x}$. Then $\overline{\gamma}$ must pass through wormhole levels of order $N_{1}$ and order $N_{2}$. However between any two wormhole levels of order $N_{1}$ and $N_{2}$, we can find all wormhole levels of order $N>N_{2}$. Hence we can modify $\overline{\gamma}$ without increasing its length by adding some extra jumps to construct a curve $\gamma$ from $p$ to $x$. Hence $d_{p}(x)\leq d_{p}(\overline{x})$ and so $d_{p}(\overline{x})=d_{p}(x)$.
	
The conclusion follows immediately from the definition of differentiability.
	
\end{proof}

Using Proposition \ref{prop_vertical_1}, Proposition \ref{prop_vertical_2} and Lemma \ref{lemma_parallel_values}, we can finally prove the main result of this section.

\begin{theorem}[Restatement of Theorem \ref{distancetheorem}]
Let $p\in F$ and denote by $B_p\subset F$ the set of points in $F$ at which $d_p$ is not differentiable. Then $h(B_{p})$ is countable, in particular $B_p$ is $\sigma$-porous.	
\end{theorem}
	
\begin{proof}
	Fix $p\in F$. As stated in Proposition \ref{vp0}, $p$ is the only point in $V_{0}^{p}$ at which $d_{p}$ is not differentiable. Now take any integer $N>0$, different from $M$ if $p$ is a wormhole of level $M$, and consider the vertical line $V_N^p$. By Proposition~\ref{prop_vertical_1}, there are only a finite number of points of non-differentiability for $d_p$ in $V_N^p$. Hence 
	\[
	S_1:=\{ x\in F: d_p(x)\mbox{ does not exist and } x\in V_N^p \text{ for some integer } N>0 \}
	\]
	is countable.
	
	Similarly, by Proposition \ref{prop_vertical_2}, the set
	\[
	S_2:=\{ x\in F:  d_p(x) \mbox{ does not exist and }x\in V_{\Delta_{N_{1},N_{2}}}^p \text{ for some integer } N_1,N_2>0 \},
	\]
	where $\Delta_{N_{1},N_{2}}=\pm\frac{2}{3^{N_1}}\pm\frac{2}{3^{N_2}}$ with signs depending on $p$, is countable.
	
	We are left to count the points of non-differentiability in vertical lines of the form $V_{\Delta_{(N_{i})}}^p$ 	where $\Delta_{(N_{i})} = \sum_i \pm\frac{2}{3^{N_i}}$ with signs depending on $p$ comes from any sequence $N_1<N_2<\ldots$ of three or more wormhole levels. In other words, the set
	\[
	S_3:=\{ x\in F: d_p(x) \text{ does not exist and }  x\in V_{\Delta_{(N_{i})}}^p \text{ for a sequence } 0<N_1<N_2<\dots \}.
	\]
	By Lemma \ref{lemma_parallel_values}, we have $h(S_3)\subseteq h(S_2)$.

	The points of non-differentiability of $d_p(\cdot)$ can be decomposed as $$B_p = p \cup S_1 \cup S_2 \cup S_3.$$ Since $h(S_3)\subseteq h(S_2)$ it follows $h(B_p)=h(p)\cup h(S_1)\cup h(S_2)$. Since both $S_1$ and $S_2$ are countable we conclude that $h(B_p)$ is countable.
	
	Finally, since $h(B_{p})$ is countable and singletons sets are porous, it follows that $h(B_{p})$ is $\sigma$-porous. Hence $B_{p}\subset h^{-1}(h(B_{p}))$
	is contained inside the preimage under $h$ of a $\sigma$-porous set, hence is $\sigma$-porous.

\end{proof}

	\section{Direct Proof of Rademacher's Theorem in Laakso Space}\label{RademacherProof}
	
	While it is well known that the Laakso space is a PI space, hence supports a differentiable structure, we were unable to find explicit justification in the literature that the differentiable structure used in this paper is the correct one. In this final section we provide this.
	
	\begin{theorem*}[Restatement of Theorem \ref{Rademacher}]
		Every Lipschitz function $f\colon F\to \bbR$ is differentiable almost everywhere.
	\end{theorem*}
	
	We divide the proof of Theorem \ref{Rademacher} into several steps.
	
	\subsection{Measure Theoretic Preliminaries}
	
	Let $q\colon I\times K\to F$ be the quotient map defined by $q(t,x)=[t,x]$. Denote by $\mathcal{H}^{1}$ and $\mathcal{H}^{Q-1}$ the Hausdorff measures on $I$ and $K$ respectively with respect to the Euclidean distance. Define the push forward measure $\mu=q_{\ast}(\mathcal{H}^{1}\times \mathcal{H}^{Q-1})$. Before giving properties of $\mu$, we first note the following simple lemma.
	
	\begin{lemma}\label{atmostone}
		Suppose $x,y\in F$ with $d(x,y)<1$. Let $N\geq 1$ be the unique integer satisfying $1/3^{N}\leq d(x,y)< 1/3^{N-1}$. Then any geodesic joining $x$ to $y$ can pass through at most one wormhole of level less than or equal to $N-1$.
	\end{lemma}
	
	\begin{proof}
		Any two wormholes of level at most $N-1$ have heights which differ by at least $1/3^{N-1}$. Since $d(x,y)<1/3^{N-1}$, a geodesic from $x$ to $y$ cannot pass through more than one such wormhole.
	\end{proof}
	
	\begin{lemma}
		The measure $\mu$ is Borel and Ahlfors $Q$-regular with respect to the metric $d$ on $F$.
	\end{lemma}
	
	\begin{proof}
		That $\mu$ is Borel follows from continuity of $q$. Fix $x=[x_{1},x_{2}]\in F$ and $0<r<1/3$. Fix $N$ such that $1/3^{N}\leq r<1/3^{N-1}$.
		
		We first estimate $\mu(B(x,r))$ from below. Without loss of generality we assume $x_{1}<2/3$, since otherwise one can apply a similar argument with up and down reversed. For each $M\geq 1$, wormholes of level $M$ are spaced apart by a distance at most $2/3^{M}$. If $M\geq N+2$ then $r/2\geq 2/3^{M}$. Hence, starting at $x$, one can reach by a curve of length at most $r$ any point $y=[y_{1},y_{2}]$ satisfying both:
		\begin{itemize}
			\item $x_{1}\leq y_{1}\leq x_{1}+r/2$, and
			\item $y_{2}$ is reached from $x_{2}$ by wormholes of level $M\geq N+2$.
		\end{itemize}
		This shows that $q^{-1}(B(x,r))$ contains a set of the form $[x_{1},x_{1}+r/2]\times K_{N+2}$, where $K_{N+2}\subset K$ is a piece of the middle third Cantor set obtained after splitting $N+2$ times. In particular it has diameter $1/3^{N+2}$. Since $\mathcal{H}^{Q-1}$ on $K$ is $(Q-1)$-regular,
		\begin{align*}
			\mu(B(x,r))&=(\mathcal{H}^{1}\times \mathcal{H}^{Q-1})(q^{-1}(B(x,r))\\
			&\geq (r/2)\mathcal{H}^{Q-1}(K_{N+2})\\
			&\geq C^{-1}r (1/3^{N+2})^{Q-1}\\
			&\geq C^{-1}r^{Q}.
		\end{align*}
		In the above estimates, $C\geq 1$ denotes a constant independent of $x$ and $r$.
		
		We now estimate $\mu(B(x,r))$ from above. By Lemma \ref{atmostone}, at vertical distance at most $r$ above and below $x$, one can find at most one wormhole of any level $M\leq N-1$. Hence $q^{-1}(B(x,r))$ is contained in a set of the form
		\[\Big( [x_{1}-r,x_{1}+r]\times K_{N-1}^{1}\Big) \cup \Big([x_{1}-r,x_{1}+r]\times K_{N-1}^{2}\Big),\]
		where $K_{N-1}^{1}, K_{N-1}^{2}$ are pieces of the middle third Cantor set obtained after splitting $N-1$ times, each having diameter $1/3^{N-1}$. This leads to the estimate
		\begin{align*}
			\mu(B(x,r))&=(\mathcal{H}^{1}\times \mathcal{H}^{Q-1})(q^{-1}(B(x,r))\\
			&\leq Cr(1/3^{N-1})^{Q-1}\\
			&\leq Cr^{Q}.
		\end{align*}
	\end{proof}
	
	By \cite[Exercise 8.11]{Hei01}, since $\mu$ and $\mathcal{H}^{Q}$ are both Ahlfors $Q$-regular on $F$ there is constant $C\geq 1$ so that
	\[C^{-1}\mathcal{H}^{Q}(E)\leq \mu(E)\leq C\mathcal{H}^{Q}(E) \qquad \mbox{for all Borel sets }E\subset F.\]
	In particular, to show sets have measure zero we may use either $\mu$ or $\mathcal{H}^{Q}$ on $F$.
	
	\begin{lemma}\label{weakFubini}
		Suppose $A\subset F$ is Borel with respect to the metric $d$ and 
		\[\mathcal{L}^{1}\{t\in I : [t,z]\in A\}=0 \mbox{ for every }z\in K.\]
		Then $\mathcal{H}^{Q}(A)=0$.
	\end{lemma}
	
	\begin{proof}
		The hypothesis implies $\mu(A)=0$ by Tonelli's theorem. Since $\mu$ and $\mathcal{H}^{Q}$ are comparable, it follows $\mathcal{H}^{Q}(A)=0$.
	\end{proof}
	
	\begin{lemma}\label{firstmeas}
		The following statements hold for every Lipschitz map $f\colon F\to \mathbb{R}$.
		\begin{enumerate}
			\item For every $z\in K$, the set
			\[D_{z}:=\{t\in I: \mbox{the directional derivative }f_{I}[t,z] \mbox{ exists}\}\]
			is Borel with respect to the Euclidean metric on $I$ and has full $\mathcal{L}^{1}$ measure.
			\item For every $z\in K$, the map from $D_{z}$ to $\bbR$ defined by $t\mapsto f_{I}[t,z]$ is Borel measurable with respect to the Euclidean metric on $I$.
			\item The set
			\[D:=\{x\in F: \mbox{the directional derivative }f_{I}(x) \mbox{ exists}\}\]
			is Borel measurable with respect to $d$ on $F$ and has full $\mathcal{H}^{Q}$ measure.
			\item The map $f_{I}\colon D\to \bbR$ defined by $x\mapsto f_{I}(x)$ is Borel measurable.
		\end{enumerate}	
	\end{lemma}
	
	\begin{proof}
		We first prove (1) and (2). Fix $z\in K$. The section $I \to \bbR$ given by $t\mapsto f[t,z]$ is Lipschitz with respect to the Euclidean metric. Differentiability of this section is equivalent to existence of the directional derivative for $f$ in $F$, except for points identified at wormhole levels which form a countable set for each section. Hence Rademacher's theorem in $I$ implies $D_{z}$ has full $\mathcal{L}^{1}$ measure. The fact that $D_{z}$ is Borel and the section is Borel measurable with respect to the Euclidean metric on $I$ can be verified with standard elementary arguments, for instance as in \cite{BS13}. This proves (1) and (2).
		
		We now prove (3) and (4). If $p=[p_{1},p_{2}]$ is not a wormhole, it is straightforward to show that $f_{I}[p_{1},p_{2}]$ exists and belongs to a closed interval $J\subset \bbR$ if and only if for all $\varepsilon \in \mathbb{Q}^{+}$ there exists $\delta\in \mathbb{Q}^{+}$ and $q\in J\cap \mathbb{Q}$ such that
		\[ \sup_{\substack{0<|t|<\delta\\ t\in \mathbb{Q}}} \frac{|f[p_{1}+t,p_{2}]-f[p_{1},p_{2}]-qt|}{t}<\varepsilon.\]
		Since $f$ is Borel measurable with respect to $d$ and the set of wormholes is countable, it follows that $D$ is a Borel measurable subset of $F$ and $f_{I}\colon D\to \bbR$ is Borel measurable map. The fact that $D$ has full $\mathcal{H}^{Q}$ measure follows from Lemma \ref{weakFubini}, because $D$ is Borel and every section of $F\setminus D$ has Lebesgue measure zero.
	\end{proof}
	
	Recall that if $g$ is a nonnegative locally integrable function on a doubling metric measure space $(X,\nu)$, then $\lim_{r\to 0} \dashint_{B(x,r)} f\dd \nu = f(x)$ for almost every $x\in X$. 
	Applying this to the characteristic function of a Borel set of locally finite measure gives a Lebesgue density theorem on $(X,\nu)$. It also follows that Borel measurable functions on $X$ are approximately continuous almost everywhere. This means that if $g\colon X\to \bbR$ is Borel measurable, then for almost every $x\in X$
	\[ \lim_{r\to 0} \frac{\nu\{y\in B(x,r): |g(y)-g(x)|>\varepsilon\}}{\nu(B(x,r))}=0 \qquad \mbox{for every }\varepsilon>0. \]
	
	We will use these facts in $\bbR$ equipped with Euclidean distance and Lebesgue measure and in $F$ equipped with the metric $d$ and Hausdorff measure $\mathcal{H}^{Q}$.
	
	\subsection{Auxilliary Sets}
	
	Fix a constant $C_{Q}\geq 1$ such that
	\[C_{Q}^{-1}r^{Q}\leq \mathcal{H}^{Q}(B(x,r))\leq C_{Q}r^{Q} \qquad \mbox{for all }x\in F \mbox{ and }0<r\leq 1.\]
	Let $f\colon F\to \mathbb{R}$ be a Lipschitz function and let $D\subset F$ denote the set of points where the directional derivative of $f$ exists. Denote $L=\mathrm{Lip}(f)$. 
	
	\begin{definition}
		Let $D'$ be the set of all points $x\in D$ which are not a wormhole. We define several sets as follows.
		\begin{enumerate}
			\item For each $\varepsilon>0$ and $x\in D$,
			\[D_{\varepsilon}(x):=\{ y\in D:|f_{I}(y)-f_{I}(x)|\leq \varepsilon\}.\]
			\item For each $\varepsilon>0$ and integer $k\geq 1$, let $E_{k}^{1}(\varepsilon)$ be the collection of all points $x=[x_{1},x_{2}]\in D'$ such that
			\[\mathcal{L}^{1}\{t\in (x_{1}-r,x_{1}+r)\cap I:[t,x_{2}]\notin D_{\varepsilon}(x)\}\leq \varepsilon r\]
			for every $0<r<1/k$.
			\item For each $\varepsilon>0$ and integer $k\geq 1$, let $E_{k}^{2}(\varepsilon)$ be the collection of all points $x\in D'$ for which
			\[\mathcal{H}^{Q}\Big( B(x,r)\setminus ( D_{\varepsilon}(x) \cap E_{k}^{1}(\varepsilon) ) \Big) <\frac{C_{Q}^{-1}\varepsilon^{Q}r^{Q}}{2^{Q+1}}\]
			for every $0<r<1/k$.
		\end{enumerate}
	\end{definition}
	
	\begin{lemma}\label{E1}
		For all $\varepsilon>0$ and integer $k\geq 1$, $E_{k}^{1}(\varepsilon)$ is Borel with respect to $d$ and
		\[\mathcal{H}^{Q}\left( F\setminus \bigcup_{k=1}^{\infty} E_{k}^{1}(\varepsilon) \right) =0.\]	
	\end{lemma}
	
	\begin{proof}
		We first show that $E_{k}^{1}(\varepsilon)$ is Borel with respect to $d$. Clearly $D'$ is a Borel measurable subset of $F$ by Lemma \ref{firstmeas}. 
		Reducing via countable intersections, it suffices to show that for every $r\in (0,1/k)\cap \mathbb{Q}$ the set of $x=[x_{1},x_{2}]\in D'$ for which
		\[ \mathcal{L}^{1}\{t\in (x_{1}-r,x_{1}+r)\cap I:|f_{I}[t,x_{2}]-f_{I}(x)|>\varepsilon \} \leq \varepsilon r\]
		is Borel measurable. Equivalently, for every $r\in (0,1/k)\cap \mathbb{Q}$ and every fixed $\alpha>0$, we must show the set
		\[ \{x\in D': \mathcal{L}^{1}\{t\in (x_{1}-r,x_{1}+r)\cap I :|f_{I}[t,x_{2}]-f_{I}(x)|>\varepsilon \} >\alpha\}\]
		is Borel measurable. However, this set can be decomposed as
		\begin{align*}
			\bigcup_{\substack{\eta>\varepsilon\\ \eta \in \mathbb{Q}}} \bigcap_{n=1}^{\infty}\bigcup_{q\in \mathbb{Q}} & \Big( \{x\in D': |f_{I}(x)-q|<1/n \} \\
			& \qquad \cap \{x\in D': \mathcal{L}^{1}\{t\in (x_{1}-r,x_{1}+r)\cap I: |f_{I}[t,x_{2}]-q|>\eta\}>\alpha\} \Big).
		\end{align*}
		Clearly $ \{x\in D': |f_{I}(x)-q|<1/n \}$ is Borel by Lemma \ref{firstmeas}. We claim
		\[(x_{1},x_{2}) \mapsto \Phi(x_{1},x_{2}):=\mathcal{L}^{1}\{t\in (x_{1}-r,x_{1}+r)\cap I: |f_{I}[t,x_{2}]-q|>\eta\}\]
		is a Borel function on $I\times K$. Indeed, it is a continuous function of $x_{1}$ and is Borel measurable in $x_{2}$ by Fubini's theorem. Hence it is Borel measurable on $I\times K$  which is a product of separable spaces \cite{AB06}. Finally we notice that
		\[\{ x\in D':\Phi(q^{-1}(x))>\alpha\}=q\{x\in q^{-1}(D'):\Phi(x)>\alpha\}.\]
		The Borel measurability with respect to $d$ follows as $q\colon q^{-1}(D')\to D'$ is bijective and continuous with continuous inverse, hence maps Borel sets to Borel sets.
		
		On each vertical line segment, the section of $\bigcup_{k=1}^{\infty} E_{k}^{1}(\varepsilon)$ has full measure with respect to $\mathcal{L}^{1}$. This follows by combining Lemma \ref{firstmeas} (1) and (2) with the fact that Borel measurable functions on $I$ are approximately continuous almost everywhere. with respect to $\mathcal{L}^{1}$. Hence Lemma \ref{weakFubini} gives the conclusion.
	\end{proof}
	
	\begin{lemma}\label{E2}
		For $\varepsilon>0$ and integer $k\geq 1$, $E_{k}^{2}(\varepsilon)$ is Borel with respect to $d$ and
		\[\mathcal{H}^{Q}\left( F\setminus \bigcup_{k=1}^{\infty} E_{k}^{2}(\varepsilon) \right) =0.\]	
	\end{lemma}
	
	\begin{proof}
		Fix $\varepsilon>0$. We first show that $E_{k}^{2}(\varepsilon)$ is Borel with respect to $d$. It suffices to show that the map $D' \to \bbR$ given by $x\mapsto \mathcal{H}^{Q}(B(x,r)\setminus (D(x)\cap E_{k}^{1}))$ is Borel. To do this we first notice that for every $\alpha>0$ the set
		\begin{align*}
			\{x\in D': \mathcal{H}^{Q}\{y\in B(x,r): y\notin E_{k}^{1} \mbox{ or }|f_{I}(y)-f_{I}(x)|>\varepsilon\}>\alpha
		\end{align*}
		can be written as
		\begin{align*}
			& \bigcup_{\substack{\eta>\varepsilon\\ \eta \in \mathbb{Q}}} \bigcap_{n=1}^{\infty}\bigcup_{q\in \mathbb{Q}} \Big(\{x\in D': |f_{I}(x)-q|<1/n\} \\
			& \qquad \qquad \qquad \cap \{x\in D': \mathcal{H}^{Q}\{y\in B(x,r): |f_{I}(y)-q|>\eta \mbox{ or }y\notin E_{k}^{1}\}>\alpha\} \Big).
		\end{align*}
		The first set inside the decomposition above is Borel by Lemma \ref{firstmeas}. The second is an open subset of $D'$, hence Borel. Hence $E_{k}^{2}(\varepsilon)$ is Borel.
		
		Using Lemma \ref{E1}, almost every point of $F$ is a density point of $E_{k}^{1}$ for some $k\geq 1$. Also, since $f_{I}$ is Borel, almost every point of $F$ is a point of approximate continuity of $f_{I}$. Hence for almost every $x$ there exists $k\in \bbN$ and $R>0$ such that
		\[\mathcal{H}^{Q}\Big( B(x,r)\setminus ( D_{\varepsilon}(x) \cap E_{k}^{1}(\varepsilon) ) \Big) <C_{Q}^{-1}\varepsilon^{Q}r^{Q}/2^{Q+1}\]
		for every $0<r<R$. Choose $K\in \bbN$ such that $K\geq k$ and $1/K<R$. Then using the fact $E_{k}^{1}(\varepsilon)\subset E_{K}^{1}(\varepsilon)$ it follows
		\[\mathcal{H}^{Q}\Big( B(x,r)\setminus ( D_{\varepsilon}(x) \cap E_{K}^{1}(\varepsilon) ) \Big) <C_{Q}^{-1}\varepsilon^{Q}r^{Q}/2^{Q+1}\]
		for every $0<r<1/K$. Hence $x\in E_{K}^{2}(\varepsilon)$. This shows $\mathcal{H}^{Q}(F\setminus \bigcup_{k=1}^{\infty} E_{k}^{2}(\varepsilon)) =0$.
	\end{proof}
	
	\subsection{Choice of Suitable Line Segments}
	
	Fix $0<\varepsilon<1$. Let $x\in \bigcup_{k=1}^{\infty} E_{k}^{2}(\varepsilon)$ and fix $K\geq 1$ such that $x\in E_{K}^{2}(\varepsilon)$. Let $y\in F$ with $d(y,x)<1/(2K)$. Let $N\geq 1$ be the unique integer such that $1/3^{N}\leq d(x,y)< 1/3^{N-1}$.
	
	Assume that infinitely many wormhole levels are needed to connect $x$ to $y$ by a geodesic. It will be clear how the following argument can be simplified if only finitely many wormhole levels or even no wormhole levels are needed. Denote $T=d(x,y)$ and choose $\gamma\colon [0,T]\to F$ such that
	\begin{itemize}
		\item $\gamma$ is a geodesic from $x$ to $y$ with $\gamma(0)=x$ and $\gamma(T)=y$.
		\item $\gamma$ is a concatenation of countably many lines in the $I$ direction parameterized at unit speed.
	\end{itemize}
	By Lemma \ref{atmostone}, any geodesic joining $x$ to $y$ must pass through at most one wormhole of level less than or equal to $N-1$. We enumerate the wormhole levels needed to connect $x$ to $y$ by a strictly increasing sequence $N_{i}$ for integer $i\geq 0$, where possibly $N_{0}\leq N-1$, but necessarily $N_{i}\geq N$ for $i\geq 1$. Since $N_{1}\geq N$ and $N_{i}$ are strictly increasing, it follows that $N_{i}\geq N+i-1$ for $i\geq 1$.
	
	For each $i\geq 0$, let $\lambda_{i}$ be the $I$ component of the point where $\gamma$ jumps using the wormhole of order $N_{i}$. Geodesics in $F$ can be chosen so that they change their direction (up or down) in the $I$ component at most twice (Proposition \ref{prop_geodesic}). Hence, during any subinterval of $[0,T]$ of length $t$, the geodesic spends at least a time $t/3$ following the same direction (either up or down but not changing between them) in the $I$ component. Since in any direction wormhole levels of order $N_{i}$ are spaced apart by at most a distance $2/3^{N_{i}}$, we can additionally choose $\gamma$ so that it satisfies:
	\begin{itemize}
		\item $\lambda_{0}\leq d(x,y)$, and
		\item $\lambda_{i}\leq 2/3^{N_{i}-1}$ for $i\geq 1$.
	\end{itemize}
	Using $N_{i}\geq N+i-1$ for $i\geq 1$ and the definition of $N$, we estimate as follows:
	\begin{align*}\sum_{i=0}^{\infty} \lambda_{i} &\leq d(x,y)+\sum_{i=1}^{\infty} \frac{2}{3^{N_{i}-1}}\\
		&\leq d(x,y)+\sum_{i=1}^{\infty}\frac{2}{3^{N+i-2}}\\
		&=d(x,y)+\frac{1}{3^{N-2}}\\
		&\leq 10d(x,y).	
	\end{align*}
	Let $(\mu_{i})_{i=0}^{\infty}$ be a strictly decreasing rearrangement of $\{\lambda_{i}:i\geq 0\}\cup \{T\}$. Thus $\mu_{0}=T$, $\mu_{i}\to 0$ as $i\to \infty$, $\gamma|_{[\mu_{i+1},\mu_{i}]}$ is a line segment for each $i\geq 0$, and
	\begin{equation}\label{summu}
		\sum_{i=0}^{\infty} \mu_{i}=\sum_{i=0}^{\infty} \lambda_{i} +T \leq 11d(x,y).
	\end{equation}
	Denote $p_{i}=\gamma(\mu_{i})$ for $i\geq 0$. Notice $p_{0}=y$ and $p_{i}\to x$ as $i\to \infty$. It follows that 
	\begin{equation}\label{telescope}
		f(y)-f(x)=\sum_{i=0}^{\infty} (f(p_{i})-f(p_{i+1})).
	\end{equation}
	Since $\gamma|_{[\mu_{i+1},\mu_{i}]}$ is a line segment in the $I$ direction, it follows $p_{i}$ is reached from $p_{i+1}$ by travelling a displacement $h(p_{i})-h(p_{i+1})$ in the $I$ direction.
	
	\subsection{Estimate Along Line Segments}
	
	Our aim is to show that $f(p_{i})-f(p_{i+1})$ is well approximated by $f_{I}(x)(h(p_{i}) - h(p_{i+1}))$ for every $i\geq 0$. Fix $i\geq 0$ until otherwise stated.
	
	\begin{lemma}\label{qs}
		There exist points $q_{i}, q_{i+1}\in F$ with the following properties:
		\begin{enumerate}
			\item $d(q_{i+1},p_{i+1})\leq \varepsilon \mu_{i+1}$,
			\item $d(q_{i},p_{i})\leq 6\varepsilon \mu_{i+1}$,
			\item $q_{i+1}\in E_{K}^{1}(\varepsilon)\cap D_{\varepsilon}(x)$,
			\item $q_{i}$ is reached from $q_{i+1}$ by travelling a vertical displacement $h(p_{i})-h(p_{i+1})$ in the $I$ direction.
		\end{enumerate}
	\end{lemma}
	
	\begin{proof}
		Fix integer $B\geq 1$ such that $1/3^B\leq \varepsilon \mu_{i+1} < 1/3^{B-1}$. This implies that within vertical distance $\varepsilon \mu_{i+1}$ above and below $p_{i+1}$, there is at most one wormhole level of order less than or equal to $B-1$.
		
		By Ahlfors $Q$-regularity of $\mathcal{H}^{Q}$,
		\[\mathcal{H}^{Q}(B(p_{i+1},\varepsilon \mu_{i+1}))\geq C_{Q}^{-1}\varepsilon^{Q}\mu_{i+1}^{Q}.\]
		Using $0<\varepsilon<1$ and $\mu_{i+1}\leq B<1/2K$ gives $\mu_{i+1}+\varepsilon\mu_{i+1}<1/K$. Hence the fact that $x\in E_{K}^{2}(\varepsilon)$ gives,
		\[\mathcal{H}^{Q} \Big( B(x,\mu_{i+1}+\varepsilon\mu_{i+1})\setminus ( D_{\varepsilon}(x) \cap E_{K}^{1}(\varepsilon)) \Big) <\frac{C_{Q}^{-1}\varepsilon^{Q}(\mu_{i+1}+\varepsilon\mu_{i+1})^{Q}}{2^{Q+1}}.\]
		Recalling $0<\varepsilon<1$, these imply
		\[ \mathcal{H}^{Q}\Big( B(x,\mu_{i+1}+\varepsilon\mu_{i+1})\setminus ( D_{\varepsilon}(x) \cap E_{K}^{1}(\varepsilon))\Big) < \frac{\mathcal{H}^{Q}(B(p_{i+1},\varepsilon \mu_{i+1}))}{2}.\]
		Since $\gamma$ is a geodesic, $d(x,p_{i+1})=d(\gamma(0),\gamma(\mu_{i+1}))=\mu_{i+1}$. Hence
		\[B(p_{i+1},\varepsilon \mu_{i+1}) \subset B(x,\mu_{i+1}+\varepsilon\mu_{i+1}).\]
		This implies
		\[ \mathcal{H}^{Q}\Big( B(p_{i+1},\varepsilon \mu_{i+1}) \setminus ( D_{\varepsilon}(x) \cap E_{K}^{1}(\varepsilon))\Big) < \frac{\mathcal{H}^{Q}(B(p_{i+1},\varepsilon \mu_{i+1}))}{2}.\]
		At least half in measure of the points in $B(p_{i+1},\varepsilon \mu_{i+1})$ are accessible from the line segment joining $p_{i+1}$ to $p_{i}$ without using a wormhole level of order at most $B-1$. Hence we can choose $q_{i+1}$ accessible from the line segment joining $p_{i+1}$ to $p_{i}$ by jump levels of order $n\geq B$ with
		\[q_{i+1} \in B(p_{i+1},\varepsilon \mu_{i+1})\cap D_{\varepsilon}(x) \cap E_{K}^{1}(\varepsilon).\]
		Clearly this choice of $q_{i+1}$ satisfies (1) and (3). 
		
		Next define $q_{i}$ from $q_{i+1}$ as stated in (4). Then $q_{i}$ can be reached from $p_{i}$ from a vertical displacement at most $2 \varepsilon \mu_{i+1}$ and wormhole levels of order $n\geq B$. Such jump levels are spaced by at most $2/3^{B}$ in the vertical direction. Hence
		\[d(p_{i},q_{i})\leq 2 \varepsilon \mu_{i+1}+(4/3^B)\leq 6 \varepsilon \mu_{i+1}.\]
		This shows that $q_{i}$ satisfies (2) and completes the proof.
	\end{proof}
	
	Note that the points $q_{i}, q_{i+1}$ in Lemma \ref{qs} may not be consistent as $i$ varies. This will not be an issue as $i$ will remain fixed while we estimate $f(p_{i})-f(p_{i+1})$. Denote $q_{i}=[t_{i},z]$ and $q_{i+1}=[t_{i+1},z]$. The fact that $q_{i+1}\in E_{K}^{1}(\varepsilon)$ gives
	\[\mathcal{L}^{1}\{t\in (t_{i+1}-r,t_{i+1}+r) \cap I:[t,z] \notin D_{\varepsilon}(q_{i+1})\}\leq \varepsilon r \qquad \mbox{for every }0<r<1/K.\]
	Since $f$ is Lipschitz, hence absolutely continuous along lines in the $I$ direction,
	\[|f(q_{i})-f(q_{i+1})-f_{I}(q_{i+1})(h(q_{i})-h(q_{i+1}))|\leq \left|\int_{t_{i+1}}^{t_{i}} (f_{I}[s,z]-f_{I}(q_{i+1}))\dd s\right|.\]
	We divide the right side into two pieces which we estimate separately. Firstly, using the definition of $D_{\varepsilon}(q_{i+1})$ gives
	\[\left|\int_{t_{i+1}}^{t_{i}} (f_{I}[s,z]-f_{I}(q_{i+1})) \chi_{\{s:[s,z]\in D_{\varepsilon}(q_{i+1})\}}(s)\dd s\right|\leq \varepsilon |t_{i}-t_{i+1}|.\]
	Secondly, the fact that $q_{i+1}\in E_{K}^{1}(\varepsilon)$ and $T<1/K$ yields
	\[\left|\int_{t_{i+1}}^{t_{i}} (f_{I}[s,z]-f_{I}(q_{i+1})) \chi_{\{s:[s,z]\notin D_{\varepsilon}(q_{i+1})\}}(s)\dd s\right|\leq 2L\varepsilon|t_{i}-t_{i+1}|.\]
	Combining the two estimates yields.
	\[|f(q_{i})-f(q_{i+1})-f_{I}(q_{i+1})(h(q_{i})-h(q_{i+1}))|\leq (2L+1)\varepsilon |t_{i}-t_{i+1}|.\]
	Using Lemma \ref{qs} together with the fact that
	\[h(p_{i})-h(p_{i+1})=h(q_{i})-h(q_{i+1})=t_{i}-t_{i+1}\] 
	yields
	\begin{align}
		&|f(p_{i})-f(p_{i+1})-f_{I}(x)(h(p_{i})-h(p_{i+1}))| \label{lineest}\\
		&\qquad \qquad \leq (2L+2)\varepsilon |h(p_{i})-h(p_{i+1})|+7L\varepsilon\mu_{i+1} \nonumber.
	\end{align}
	
	\subsection{Conclusion of the Proof}
	
	Adding \eqref{lineest} over all integers $i\geq 0$ using \eqref{summu}, \eqref{telescope} and the triangle inequality gives
	\begin{align*}
		&|f(y)-f(x)-f_{I}(x)(h(y)-h(x))|\\
		&\qquad \qquad \leq (2L+2)\varepsilon \sum_{i=0}^{\infty}|h(p_{i})-h(p_{i+1})|+7L\varepsilon\sum_{i=0}^{\infty}\mu_{i+1}\\
		&\qquad \qquad \leq (2L+2)\varepsilon d(x,y) + 77L\varepsilon d(x,y).
	\end{align*}
	To summarize, we have shown the following. Given $0<\varepsilon<1$, for almost every $x\in F$ there exists $\delta>0$ such that $d(x,y)<\delta$ implies
	\[|f(y)-f(x)-f_{I}(x)(h(y)-h(x))|\leq (79L+2)\varepsilon d(x,y).\]
	For each integer $n\geq 2$, 
	let $D_{n}$ be the set of $x\in F$ such that
	\[ \limsup_{y\to x} \frac{|f(y)-f(x)-f_{I}(x)(h(y)-h(x))|}{d(x,y)}>\frac{(79L+2)}{n}.\]
	We have shown that $\mathcal{H}^{Q}(D_{n})=0$ for each integer $n\geq 2$. Hence $\mathcal{H}^{Q}(\cup_{n=1}^{\infty}D_{n})=0$. Since $(\cup_{n=1}^{\infty}D_{n})^{c}=\cap_{n=1}^{\infty}D_{n}^{c}$ is the set of points where $f$ is differentiable, it follows that $f$ is differentiable almost everywhere. This proves Theorem \ref{Rademacher}.

\end{document}